\documentclass[a4paper,11pt]{article}

\usepackage{amsfonts}
\usepackage{amsmath,amsthm,amssymb,amsfonts, mathrsfs, epsfig,latexsym,graphicx}
\usepackage{color}

\newtheorem{thm}{Theorem}[section]
\newtheorem{Remark}{Remark}[section]
\theoremstyle{definition}

\theoremstyle{Remark}

\newtheorem{lemma}[thm]{Lemma}
\theoremstyle{lemma}
\newtheorem{corollary}[thm]{Corollary}
\theoremstyle{corollary}
\newtheorem{proposition}[thm]{Proposition}
\theoremstyle{proposition}
\numberwithin{equation}{section}

\def\O{\Omega}
\def\o{{\omega}}

\def\F{{\mathcal F}}
\def\l{\lambda}
\def\p{\partial}
\def\s{{\sigma}}

\def\e{\epsilon}
\def\a{\alpha}
\def\d{\delta}
\def\th{\theta}
\def\R{{\mathbb{R}}}

\def\sni{\smallskip\noindent}

\def\pr{^\prime}

\def\dpr{^{\prime\prime}}

\providecommand{\norm}[1]{\lVert#1\rVert}
\begin{document}
\title{The Enskog Process}
\author{S. Albeverio\footnote{Institute of Applied Mathematics and HCM, BiBoS, IZKS, University of  Bonn, Germany. Email: albeverio@iam.uni-bonn.de}
,  B. R\"udiger\footnote{Bergische Universit\"at Wuppertal Fakult\"at 4 -Fachgruppe Mathematik-Informatik,
Gauss str. 20, 42097 Wuppertal, Germany. Email: ruediger@uni-wuppertal.de}
 and P. Sundar \footnote{Department of Mathematics, Louisiana State University, Baton Rouge, La 70803, USA. Email: sundar@math.lsu.edu}} 
\date{ }
\maketitle
\begin{abstract}
\noindent   The existence of a weak solution to a McKean-Vlasov type stochastic differential system corresponding to the Enskog equation of the kinetic theory of gases  is established under suitable hypotheses. The distribution of  any solution to the system at each fixed time is shown to be unique. The existence of a probability density for the time-marginals of  the velocity is verified  in the case where the initial condition is  Gaussian, and is shown to be the density of an invariant measure. \\


\end{abstract}
\section{Introduction}
The Boltzmann equation describes the time evolution of the density function in a phase (position-velocity) space for a classical particle (molecule) 
under the influence of  other particles in a diluted (or rarified) gas \cite{Bo} (evolving in  vacuum for a given  initial distribution). It forms the basis for 
the kinetic theory of gases, see, for e.g. \cite{Ce4}. 

\par
If $f$ is the density function, which depends on time $t \ge 0$, the space variable $x \in \R^d$, and the velocity variable $u \in \R^d$ of the point particle, 
then $f(t, x, u)dx\,du$ is by definition the probability for the particle to have position $x$ in a volume element $dx$ around $x$ and velocity $u$ within the volume element $du$ around $u$. For a single type of particles all of mass $m > 0$, in the absence of external forces, the Boltzmann equation has the general form 
\begin{equation}
\frac{ \partial f}{\partial t} (t, x, u)+ u \cdot \nabla_x f(t, x, u) = Q(f, f)(t,x,u), \label{eqn1.1}
\end{equation}
where $Q$ is a certain quadratic form in $f$, called collision operator (or integral). 

\par 
Set $\Lambda := \R^3\times (0, \pi] \times [0, 2\pi)$. Then $Q$ can be written in the general form 
\begin{equation}
Q(f, f) (t,x,u)= \int_{\Lambda} \{f(t, x, u^\star) f(t, x, v^\star)-f(t, x, u) f(t, x, v)\}B(u, dv, d\th)d\phi. \label{eqn1.2}
\end{equation}
It is assumed that any gas particle travels straight until an elastic collision occurs with another particle.  Each $v \in \R^3$ in (\ref{eqn1.2}) denotes the velocity of an incoming  particle which may hit, at the fixed location $x\in \mathbb{R}^3$,  particles whose velocity is fixed as 
$u\in \mathbb{R}^3$. 
Let $u^\star \in \R^3$ and $v^\star \in \R^3$ denote the resulting outgoing velocities corresponding to the incoming velocities $u$ and  $v$ respectively. 
$\theta$ $\in (0,\pi]$ denotes the azimuthal or colatitude angle of the deflected velocity, $v^{\star}$ (see \cite{Ta}).  Having determined $\theta$, 
the longitude angle $\phi \in [0, 2\pi)$ measures in polar coordinates,  the location of  $v^*$, and hence  that of $u^\star$, as explained below. 
 
\noindent In the Boltzmann model as the collisions are assumed to be elastic,  conservation of kinetic energy as well as momentum of the molecules holds, i.e. 
 considering particles of mass $m =1$, the following equalities hold:
\begin{equation}
\left\{
\begin{aligned}
u^\star+v^\star =u+v\\
(u^\star)^2+(v^\star)^2 =u^2+v^2 
\end{aligned}
\right. \label {eqn1.3}
\end{equation} 
\begin{equation} 
\left\{
\begin{aligned}
v^\star=v+ ({\bf n},u-v) {\bf n}\\
u^\star=u-({\bf n},u-v){\bf n}
\end{aligned}
\right. \label {eqn1.4}
\end{equation}
   
\noindent where 
\begin {equation}\label {eqn1.5}
{\bf n}=\frac{v^\star-v}{|v^\star-v|}
\end {equation}
 where $(\cdot, \cdot)$ denotes the scalar product, and $| \cdot |$, the Euclidean norm in $\mathbb{R}^3$.

\noindent \begin{Remark} \label {Remoutcomevel} The Jacobian of the transformation (\ref {eqn1.4}) has determinant 1 and  $(u^\star)^\star=u$ since the collision dynamics are reversible.\end {Remark}

\noindent The outgoing velocity  $u^*$ is then uniquely determined in terms of the colatitude angle $\theta\in (0,\pi]$ measured from the center, and longitude angle $\phi \in  [0,2\pi)$ of  the deflection  vector ${\bf n}$ in the sphere with northpole $u$ and southpole $v$ centered at $\frac{u+v}{2}$, which are used in equation (\ref{eqn1.1}) and \eqref{eqn1.2} (see  e.g. \cite {Br},  \cite {HK}, \cite {Ta2}).

\par 
$B(u, dv, d\th)$ is a $\sigma$-additive positive measure defined on the Borel $\sigma$-field ${\mathcal B}(\R^3) \times {\mathcal B}((0, \pi])$, depending 
 (Lebesgue) measurably on $u \in \R^3$. The form of $B$ depends on the version of Boltzmann equation one has in mind. In Boltzmann's original work 
 \cite{Bo}, 
\begin{align}
B(u, dv, d\th) &= |(u-v)\cdot n| dv d\th \notag\\
& = |u - v| dv \cos \left(\frac{\th}{2}\right) \sin \left(\frac{\th}{2}\right) d\th,\label{eqn1.6}
\end{align}
where in \eqref{eqn1.6} we used that $\frac{\pi}{2}- \frac{\th}{2}$ is the angle between $u-v$ and ${\bf n}$, so that 
\begin{equation} |(u-v,{\bf n})|= |u-v|\cos(\frac{\pi}{2}- \frac{\th}{2})= |u-v| \sin(\frac{\th}{2}), \label{eqscalar}\end{equation}
\begin{equation} \;\;\text{and} \;\;\;|(u-v,{\bf n})|d{\bf n}=|u-v| \sin(\frac{\th}{2})\cos(\frac{\th}{2})d\th d\phi=B(u, dv, d\th)d\phi\label{eqcross}\end{equation} is the differential cross section scattering the  velocities $v$ of incoming particles colliding with the particle with velocity $u$, written in polar coordinates.

\par
In the case where the molecules interact by a force which varies as the $n$th inverse power of the distance between their centers, one has 
\cite{Ce4}, 
\begin{equation}
B(u, dv, d\th) = |(u-v)|^{\frac{n-5}{n-1}} \beta(\th)dvd\th \label{eqn1.7}
\end{equation}
where $\beta$ is a Lebesgue measurable positive function of $\th$. In particular, for $n = 5$, one has the case of ``Maxwellian molecules", where 
\[
B(u, dv, d\th) = \beta(\th) dv d\th.
\]
The function $\beta(\th)$ decreases and behaves like $\th^{- 3/2}$ for $\th \downarrow 0$, see, for e.g. \cite{Ce4}. Note that 
in the latter case $\int_0^{\pi}\beta(\th)d\th =+ \infty$. We note that for Maxwellian particles the cross section $B(u, dv, d\th)d\phi$ does not depend on the modulus $|u-v|$ of the velocity difference between the velocity $u$ of the particle and the velocities $v$ of incoming particles.

\par In the present paper, we shall mainly assume that
\begin{equation}
B(u, dv, d\th) = \sigma (|u - v|) dv Q(d\th) \label{eqn1.8}
\end{equation}
where  $Q$ is a $\sigma$-finite measure on ${\mathcal B}((0, \pi])$, and $\sigma$ is a bounded, Lipschitz continuous, positive function on $\R^+$.  The assumptions on $\sigma$, though an improvement on the existing results for the case $\sigma =1$, do  restrict applicability to physically realizable molecules.
When $Q(d\th)$ is taken to be integrable, one speaks of a cut-off function.

\noindent
\begin{Remark} \label{Rem1.2}
Both the rigorous derivation of Boltzmann equation from a microscopic model, and the study of existence, uniqueness and properties of solutions of the Boltzmann equation still present many challenging and open problems. For the derivation problem, see, for e.g., \cite{CIP}.  
\end{Remark}

Morgenstern \cite{Mo} ``mollified" $Q(f, f)$ by replacing it by 
\begin{align*}
&Q_M(f, f)(t,x,u) \\
& = \int \{f(t, x, u^\star) f(t, y, v^\star)-f(t, x, u) f(t, y, v)\}K_M(x, y)B(u, dv, d\th) dy d\phi
\end{align*}
with some measurable $K_M$ and $B$    such that $K_M(x, y)B(u, dv, d\th)$ has a bounded density with respect to Lebesgue measure $dv \times d\th$, and 
obtaining a global existence theorem in $L^1(\R^3 \times \R^3)$. Povzner \cite{Pov} obtained existence and uniqueness in the space of Borel measures in $x$, with  the term $K_M(x, y)B(u, dv, d\th) dy d\phi$ replaced by $K_P(x-y, u-v)dv dy$ (with  a suitable reinterpretation of the relations between 
$x, u^*, v^*$ and $y, u, v$, and suitable moments assumptions on $K_P$). 

\par
According to \cite{Ce4} (p. 399), this modification of Boltzmann's equation by Povzner is ``{\it close to physical reality}". Cercignani also notes that Povzner equation has a form similar to the {\it Enskog equation} for dense gases, which we shall discuss below.  Modification in  another direction consists in taking the space of velocities  as discrete, and is discussed in \cite{Ce4} (pp. 399-401). 

\par
The majority of further mathematical results concerns the spatially homogeneous case, where the initial condition on $f$ is assumed to be independent of the space variable $x$, so that at all times $f$ itself does not depend on  $x$. For such results see, e.g. \cite{Ce4}, \cite{Ce7}. 

\par
Let us now associate to \eqref{eqn1.1}, \eqref{eqn1.2} its weak (in the functional analytic sense) version. The following proposition is instrumental in this direction.

\par
Using also  Remark \ref{Remoutcomevel}, Tanaka  \cite {Ta} proved the following result that is important for the weak formulation of the equation.
\begin {proposition} \label {PropAppTanaka}  
Let $\Psi(x,u) \in C_0(\mathbb{R}^6)$, as a function of $x\in \mathbb{R}^3$, $u\in \mathbb{R}^3$.  With $B$ as in \eqref{eqn1.8}, we have
\begin {equation} 
\begin {split}
&\int_{\mathbb{R}^9\times (0,\pi]\times [0,2 \pi)} \Psi(x,u)f(t, x, u^\star) f(t, x, v^\star) B(u, dv, d\th) dx du  d\phi
\\&= \int_{\mathbb{R}^9\times (0,\pi]\times [0,2 \pi)} \Psi(x,u^\star)f(t, x, u) f(t, x, v) B(u, dv, d\th) dx du d\phi \label{eqn1.9}
\end {split}
\end {equation}
\end {proposition}
The above result  is proven using equation \eqref{eqn1.4} and Remark \ref{Remoutcomevel} \cite{Ta}.
From now on, we will assume that $B$ is as in \eqref{eqn1.8}. 

\smallskip \noindent {\bf Weak formulation of the Boltzmann equation}

Consider the Boltzmann equation \eqref{eqn1.1} with collision operator \eqref{eqn1.2}. We multiply \eqref{eqn1.1} by a function  $\psi$  (of $(x, u) \in \mathbb{R}^6$)
 belonging to  $C^1_0(\mathbb{R}^6)$, 
and integrate with respect to $x$ and $u$. Using integration by parts and Proposition \eqref{PropAppTanaka}, we arrive at the weak form of the Boltzmann equation: 
\begin{align}
&\int_{\mathbb{R}^6} \psi(x,u) \frac{ \partial f}{\partial t} (t, x, u)dxdu - \int_{\mathbb{R}^6} f(t, x, u)(u,  \nabla_x \psi(x,u)) dxdu \notag\\
&= \int_{\mathbb{R}^6} f(t, x, u) L_f\psi(x,u) dxdu \label {WBE}
\end {align}
 for all $t\in \mathbb{R}_+$  with
\[
L_f\psi(x,u)= \int_{\mathbb{R}^3\times (0,\pi]\times [0,2 \pi)} \{\psi(x,u^\star)- \psi(x,u)\}f(t, x, v)B(u, dv, d\th)d\phi,
\]
where $B$ is as in \eqref{eqn1.8}.

\par
To proceed further, let us introduce an  approximation to the weak form of the Boltzmann equation by introducing a smooth real-valued function $\beta$  
(which should not be confused with the one appearing in \eqref{eqn1.7}) with compact support defined on $\R^1$:
 \begin{align}
& \int_{\mathbb{R}^6} \psi(x,u) \frac{ \partial f}{\partial t} (t, x, u)dxdu - \int_{\mathbb{R}^6} f(t, x, u)(u,  \nabla_x \psi(x,u)) dxdu\notag \\
 & = \int_{\mathbb{R}^6} f(t, x, u) L^{\beta}_f\psi(x,u) dxdu 
 \label{WBEA} 
 \end{align}
 for all $\psi \in C^1_0(\mathbb{R}^6)$ and for all $t\in \mathbb{R}_+ $ with $L^{\beta}_f\psi(x,u) $
\[
= \int_{\mathbb{R}^6\times (0,\pi]\times [0,2 \pi)} \{\psi(x,u^\star)- \psi(x,u)\}f(t, y,v) \beta(|x - y|)dyB(u, dv, d\th) d\phi . 
\]
Heuristically, when $\beta \to \delta_0$, then any solution of \eqref{WBEA} tends to a solution of Boltzmann's equation \eqref{WBE}, so that 
$\beta$ can be seen as a regularization for \eqref{WBE}. 

\par
Equation \eqref{WBEA} is thus the (functional analytic) weak form of an equation closely related to the Boltzmann equation, which can be written as 
\begin{equation}
\frac{\p f}{\p t} f(t, x, u) + u \cdot \nabla_xf(t, x, u) = Q_E^{\beta}(f, f) (t,x,u), \label{ens}
\end{equation}
with 
\begin{align*}
&Q_E^{\beta}(f, f) (t,x,u) \\
& = \int_{\Lambda} \int_{\R^3} \{f(t, y, u^*)f(t, x, v^*) - f(t, y, u)f(t, x, v)\} \beta(|x - y|)dy B(u, dv, d\th)d\phi.
\end{align*}

In the case where $\beta$ is replaced by the characteristic function (or a smooth version of it like in \cite{Ce4}) of a ball of radius $\epsilon > 0$, this is Enskog's equation used for (moderately) ``dense gases" taking into account interactions at distance $\epsilon$ between molecules. For Enskog's equation, see e.g. \cite{Ce4}, \cite{Ce3}, \cite{PS} (pp. 6, 14), \cite{En1}, \cite{En2}, \cite{Ar4}, 
\cite{AC2}, \cite{Ce5}, \cite{Po2}, \cite{BL}, \cite{BT2}. For versions of the equation in a bounded region, see \cite{AC}, \cite{Po1}. 
The relationship between the Enskog and the Boltzmann equation have been discussed in several publications. In particular, their asymptotic equivalence (with respect to the support of $\beta$ shrinking to $\{0\}$) has been discussed in \cite{BL}. In \cite{PS}, 
a pointwise limit has been established. 

\par
If $\mu_t$ denotes the Borel probability measure on $\R^6$ corresponding to a smooth density function $f(t,x, u)$, 
i.e.
\[
\mu_t(dx, du) = f(t, x, u) dx du,
\] 
then the equation \eqref{WBEA} can be 
written as
\begin{equation}
\frac{\p}{\p t} \langle \mu_t, \psi \rangle - \langle \mu_t, (u, \nabla_x \psi(x,u))\rangle = \langle \mu_t, L^{\beta}_{\mu_t}\psi \rangle \label{MtWBEA}
\end{equation}
where $L^{\beta}_{\mu_t}\psi(x,u) $
\[
= \int_{\mathbb{R}^6\times (0,\pi]\times [0,2 \pi)} \{\psi(x,u^\star)- \psi(x,u)\} \beta(|x - y|)\mu_t(dy, dv)B(u, dv, d\th)d\phi.  
\]
In the above, we have used the sharp bracket $\langle \cdot, \cdot \rangle$ to denote integration with respect to $\mu_t$ while $(\cdot , \cdot )$ denotes the inner product in $\R^3$. If $\mu_t$ satisfies \eqref{ens}, we say that $\mu_t$ is a weak solution of the Enskog equation.

\medskip \noindent
 Define the space  $\mathbb{D}:= \mathbb{D}(\mathbb{R}_+,\mathbb{R}^3)$ as the  space of all  right continuous functions with left limits defined  on 
$[0, \infty)$ taking values in $\mathbb{R}^3$, and equipped with the topology induced by the Skorohod metric (see e.g. \cite {Bi}). 
We denote the value of any $\omega \in \mathbb{D}$ at any time $s$ by  $\omega_s$ or $\omega(s)$.  Likewise,  the time marginal of a Borel probability measure 
$\mu$ on $\mathbb{D}$ will be denoted by $\mu_s$ for all $s \in [0, \infty)$. The measure $\mu_s$ will be a Borel probability measure on $\mathbb{R}^3$. We will use similar notations for functions in $\mathbb{D}\times \mathbb{D}$ and for Borel measures on $\mathbb{D}\times \mathbb{D}$. 

\medskip
\noindent
If the measure $\mu_t$ in \eqref{MtWBEA} is the marginal at time $t$ of a Borel probability measure $\mu$ on $\mathbb{D}\times \mathbb{D}$, then we can write the Enskog equation 
\eqref{MtWBEA} as follows:
\begin{equation}
\frac{\p}{\p t} \langle \mu, \psi(x_t, u_t) \rangle - \langle \mu, (u_t, \nabla_{x_t} \psi(x_t, u_t))\rangle = \langle \mu, L^{\beta}_{\mu}\psi(x_t, u_t) \rangle \label{MWBEA}
\end{equation}
where $x_t, u_t$ are $t$-coordinates in $\mathbb{D}\times \mathbb{D}$ and $L^{\beta}_{\mu}\psi(x_t, u_t)$
\[
= \int_{U_0} \{\psi(x_t, u_t^\star)- \psi(x_t, u_t)\}\s(|u_t - v_t|) \beta(|x_t - y_t|)\mu(dy, dv)Q(d\theta) d\phi,  
\]
where we used the form \eqref{eqn1.8} of $B$ and the notation $U_0$ for $\mathbb{D}\times \mathbb{D}\times (0,\pi]\times [0,2 \pi)$.
In the following, we formulate the connection between equation \eqref{MWBEA} and stochastic analysis.  
\begin{Remark}
The idea of looking at solutions of certain deterministic nonlinear parabolic evolution equations in connection with probability measures describing the distributions of suitable associated Markov processes goes back to McKean \cite{McK}. For a spatial homogeneous version of our present context  for the case $\sigma=1$,  this idea has been adapted and ingeniously implemented by Tanaka \cite{Ta},\cite{Ta3}, and successively developed, 
for this case, e.g., in \cite{Ta1}, \cite{Ta2},\cite{Ta3}, \cite{Fu},  \cite{FG}, \cite{HK}. In our work, we avoid the assumption of spacial homogeneity, and we allow $\sigma$   to depend on  $|u-v|$. 
\end{Remark}
Let us first derive heuristically  the evolution of the stochastic process $(X_s, Z_s)_{s\in \mathbb{R}}$,  describing the evolution of position and velocity of a  particle evolving according to the Enskog equation \eqref{MWBEA}.  In the present context, the evolution of the velocity $(Z_s)_{s\in \mathbb{R}}$ of one particle is obtained  by integrating (or in other words "summing") the velocity displacements $\alpha(Z_s,v_s.\theta,\phi)$ with respect to a  counting measure ${N}_{X, Z}( ds,dy, dv,d\theta,d\phi)$  which  depends over a time interval   $ds$ on the distribution $\mu(dy, dx)$ in position and velocity of the gas particles, as well as the position and velocity $(X_s, Z_s)$ of the particle itself because of the presence of  particles being close enough to hit (guaranteed by the function $\beta$), and the scattering measure for the velocity $B(u,dv,d\theta)d\phi$, defined in \eqref{eqn1.8}. The position then evolves according to $X_t= X_0 + \int_0^t Z_s ds $. Let us introduce such a suitable jump-Markov process $(X_s, Z_s)_{s\in \mathbb{R}}$.
Let $\mu(dx,dv)$ be a  probability measure  on  $\mathbb{D} \times \mathbb{D}$.  Let  $\tilde {N}_{X, Z}( ds,dy, dv,d\theta,d\phi)$ be a compensated  random measure (crm) defined on a filtered probability space $(\Omega,\mathscr{F},\mathscr{F}_t,\mathrm{P})$ with compensated measure (or simply, compensator), 
\begin{equation} 
d\Gamma:=\Gamma(dy,dv,d\theta,d\phi,ds)  = \s(|Z_s - v_s|)\beta(|X_s - y_s|)\mu(dy,dv) Q(d\theta) d\phi ds
\label{eqcrm} 
\end{equation}
 on 
$\mathbb{D}\times \mathbb{D}\times (0,\pi]\times [0,2 \pi)\times \mathbb{R}_+$  (where  we recall that $Q(d\theta) $ is a $\s$ - finite measure on the Borel $\sigma$ - algebra $\mathcal {B}((0,\pi])$,  $d\phi$ is the Lebesgue measure on $\mathcal {B}([0,2 \pi))$, and $\beta$ is a $C_0^{\infty}({\mathbb R}^1)$ function with support near $0$. Here, $X, Z$ are elements of $\mathbb{D}$, and $v_s, y_s$ are 
$s$-coordinates of $v, y$ in $\mathbb{D}$.   

\par
Now, let  $(X_s,Z_s)_{s \in \mathbb{R}_+}$ be the process defined below, taking values in the Skorohod space  $\mathbb{D}\times \mathbb{D}$  with 
  the joint distribution of $(X_s,Z_s)_{s \in \mathbb{R}_+}$ denoted by $\mu(dx, dz)$:  
\begin{align}
Z_t= &Z_0 + \int_0^t \int_{\mathbb{D}\times\mathbb{D}\times (0,\pi]\times [0,2 \pi)} \alpha(Z_s, v_s, \theta, \phi) 
\tilde {N}_{X, Z}( ds,dy, dv,d\theta,d\phi) \notag\\
&+ \int_0^t \int_{\mathbb{D}\times \mathbb{D}\times (0,\pi]\times [0,2 \pi)} \alpha(Z_s, v_s, \theta, \phi) \s(|Z_s - v_s|)\beta(|X_s - y_s|) \notag\\
&\,\,\;\;\;\;\;\;\;\;\;\;\;\;\;\;\;\;\;\;\;\;\;\;\;\mu(dy, dv) Q(d\theta) d\phi ds\label{BM}\\
X_t=& X_0 + \int_0^t Z_s ds  \label{BM2}
\end{align}
The initial values $X_0$ and $Z_0$ are specified. 
We have set 
\noindent  \begin {equation} \label {kernel}
\alpha(u,v,\theta,\phi):= ({\bf n},u-v) {\bf n}, 
\end {equation}

\noindent where, as  above,  the deflection  vector ${\bf n}$ is given in spherical coordinates, i.e. in terms of  the colatitude angle $\theta\in (0,\pi]$ and longitude angle $\phi \in  [0,2\pi)$. 

\par We have obtained such  a process heuristically considering the physics governing the evolution of the particles and will prove in this article that this is the stochastic process whose  law corresponds  to the solution of the Enskog equation \eqref{MWBEA}. We remark however  that the stochastic equation  \eqref{BM},\eqref{BM2} is defined in terms of a counting measure $\tilde {N}_{X,Z}$  with random compensator $\s(|Z_s - v_s|)\beta(|X_s - y_s|)\mu(dy,dv) Q(d\theta) d\phi ds$. The mathematical  theory of point processes with random compensator has been analyzed extensively in e.g. \cite{IW}, or \cite {JS}, but as the theory of Stochastic Differential equations with Poisson random measure  is more developed and better known, we prefer here to rewrite  \eqref{BM},\eqref{BM2} in an equivalent stochastic equation written in terms of a stochastic integral w.r.t to a Poisson random measure associated to a L\'evy process, which is the following:
\begin{equation}
 \begin {split}
   Z_t &=  Z_0 + \int_0^t \int_{U} \alpha(Z_s, v_s, \theta, \phi)1_{[0,\; \s(|Z_{s} - v_s|)\beta(|X_s - y_s|)]}(r) 
d\tilde {N}^\mu\\& + \int_0^t \int_{U_0} {\alpha}(Z_s, v_s, \theta, \phi)\s(|Z_s - v_s|)\beta(|X_s - y_s|) \mu(dv, dy) Q(d\theta) d\phi  ds \label{eq-velb}
 \end {split}
 \end{equation}
 \begin{equation}
 \label{eq-spaceb}
  X_t = X_0 + \int_0^t Z_sds, 
 \end{equation}
where $\mu$ is still the law of the process $(Z_t, X_t)$, $t \ge 0$, but now  $\tilde {N}^\mu( dy,dv, d\theta d\phi,dr,ds)$ is a compensated Poisson random measure (cPrm) with Poisson measure ${N}^\mu:=$${N}^\mu( dy,dv, d\theta d\phi,dr,ds)$ and compensator $\mu(dy,dv) Q(d\theta) d\phi dr ds $ on $\mathbb{D} \times \mathbb{D} \times (0, \pi]\times [0, 2\pi)\times [0,1]\times \mathbb{R}_+ $ .

\par That \eqref{eq-velb},\eqref{eq-spaceb} and  \eqref{BM},\eqref{BM2} are equivalent equations can be shown with at least two different methods:
\begin{itemize}
\item[i)] For each Borel-subset $B$ of $\mathbb{D}\times \mathbb{D}\times (0,\pi]\times [0,2 \pi)$,  the counting measure $ {N}_{X,Z}(B\times [0,t))$ can be represented by 
\begin{equation*} \int_0^t\int_B\int_{[0,1]} 1_{[0,\; \s(|Z_{s} - v_s|)\beta(|X_s - y_s|)]}(r) 
d {N}^\mu \,.\end{equation*} This is a consequence of the following equation, which shows  the relation between  the random  compensator $\Gamma$ of the point measure ${N}_{X, Z}$ defined  \eqref{eqcrm} and the compensator $\mu(dy,dv) Q(d\theta) d\phi ds$ of the Poisson random measure  $ {N}^\mu$:
\begin{align*} &\Gamma(B\times [0,t))= \int_0^t\int_B \s(|Z_s - v_s|)\beta(|X_s - y_s|)\mu(dy,dv) Q(d\theta) d\phi ds \\&
=\int_0^t\int_B\int_{[0,1]} 1_{[0,\; \s(|Z_{s} - v_s|)\beta(|X_s - y_s|)]}(r)\mu(dy,dv) Q(d\theta) d\phi ds dr \end{align*}

\item[ii)] Given  $\psi \in C^2_0(\R^6)$, for all  $t \ge 0$,  $\Delta t \ge 0$, the It\^o formula  $\psi(X_{t+ \Delta t},Z_{t+ \Delta t})- \psi(X_t,Z_t)$  for  \eqref{eq-velb},\eqref{eq-spaceb} and  \eqref{BM},\eqref{BM2} can be proven to be exactly the same.  This implies that  \eqref{eq-velb}, \eqref{eq-spaceb} and  \eqref{BM},\eqref{BM2}  solve the same martingale problem and are the same process in weak sense. In particular the law of the corresponding process  $(Z_t, X_t)$, $t \ge 0$  solves in both cases  the  Enskog equation, as proven in Proposition \ref{PropEns} for  \eqref{eq-velb},\eqref{eq-spaceb}. The  It\^o formula for \eqref{eq-velb}, \eqref{eq-spaceb} is computed in the proof of Proposition \ref{PropEns}. It can be obtained in a similar way for \eqref{BM},\eqref{BM2} by using Theorem 2.42, Ch. II in \cite {JS}. 
\end{itemize}

\par
 We call the process $(Z_t, X_t)$, $t \ge 0$, given by \eqref{eq-velb},\eqref{eq-spaceb}, \eqref{kernel}(resp. its law $\mu = \mu_t, t \ge 0$) as the Markov process (resp. law) associated with the Enskog equation described by \eqref{MWBEA}.
Its existence  is  proven in Theorem \ref{thm-existence-vel} in  Section 2  under suitable conditions which are satisfied by some physical models. In Proposition \ref{PropEns} we will  prove that for any finite fixed time $T > 0$,  its law  $\mu = \{\mu_t\}, 0 \le  t \le T$ solves 
the Enskog equation (\ref{MWBEA}). Uniqueness of  the Markov process $(Z_t, X_t)$, $t \ge 0$ solving \eqref{eq-velb},\eqref{eq-spaceb}, \eqref{kernel}, is proven in Theorem \ref{ThmUniq} in Section 3 for the time marginals. The existence of a density $f(t,x,z)$ for the distribution $\mu = \{\mu_t\}, 0 \le t \le T$  is proven in Section 4, for the particular case where the velocity marginals are time invariant. $f(t,x,z)$ solves then the Enskog equation \eqref{WBEA}

 It is worthwhile to mention that we have not made the assumption of space homogeneity. We allow $\sigma$ that appears as the differential cross section (see equation \eqref{eqn1.8}) to depend on $|u-v|$. 
\section{Existence Results}
In this section we establish the existence  of a solution  of the system of stochastic equations \eqref{eq-velb},\eqref{eq-spaceb}, with $(\theta, \phi)$ denoted by $\xi$ that takes values in the set $\Xi := (0, \pi]\times [0, 2\pi)$. Also, 
$Q(d\theta)d\phi$ is written as $Q(d\xi)$ for notational simplicity. From the physical model, we know that  $Q(d\xi) $ should be a $\s$-finite measure, and hence taken as $\s$-finite.

\smallskip
\noindent
{\bf Hypotheses A:}

\begin{itemize}
\item[\bf{A1.}] The measure $Q$ is finite outside any neighborhood of $0$, and for all $\e > 0$, $Q$ satisfies 
\[
 \int_0^{\e} \th Q(d\th) < \infty.
\]
\item[\bf{A2.}] $\s: \R^+ \to \R^+$ (as entering \eqref{eqn1.8}) is a bounded, Lipschitz continuous function on $\R^+$.
\end{itemize}

\par There are many useful consequences of {\bf A1}. Recall that $\alpha (z, v, \xi) = ({\bf n} \cdot (z-v)){\bf n}$ with $\frac{\pi}{2} - \frac{\th}{2}$ as the angle between the vectors $(z-v)$ and $\bf{n}$ (see \eqref{eqscalar}). Hence, condition {\bf A1} implies that there exists a constant $C$ such that the following estimates hold.

\begin{equation} 
\int_{\Xi} |{\alpha}(z,v, \xi) - {\alpha}(z',v', \xi)|^2Q(d\xi) \leq C (|z-z'|^2+|v-v'|^2) \label{B1}
\end{equation}
\begin{equation}\int_{\Xi} |{\alpha}(z,v, \xi) - {\alpha}(z',v', \xi)|Q(d\xi) \leq C (|z-z'|+|v-v'|).\label{B3} 
\end{equation}
From \eqref{B1}, it follows, by setting  $z'$ and $v'$ to be $z$, and using the fact that  $\a(z, z, \xi) = 0$ for all $\xi \in \Xi$, $z \in \R^3$ that 
\begin{equation}
\int_{\Xi} | {\alpha}(z,v, \xi)|^2 Q(d\xi) \leq C|z-v|^2,\label{BB2}
\end{equation}
and hence
\begin{equation}
\int_{\Xi} | {\alpha}(z,v, \xi)|^2 Q(d\xi) \leq C(|z|^2+|v|^2).\label{B2}
\end{equation}
In a similar way, from \eqref{B3} it follows 
\begin{equation}
\int_{\Xi}  |{\alpha}(z,v, \xi)| Q(d\xi) \leq C|z-v|,\label{growth-L}
\end{equation}
and hence
\begin{equation}
\int_{\Xi}  |{\alpha}(z,v, \xi)| Q(d\xi) \leq C(|z|+|v|),\label{growth}
\end{equation}

\par
 Condition {\bf A2}  on $\s$ is required for mathematical reasons.  It is worthwhile to note that for Maxwellian molecules,  $\s$ is the constant function  identically equal to $1$.  Hence hypothesis {\bf A2}  leads to more mathematical generality  but still falls short of physical reality. 

Before we proceed further,  we recall the following: Since the function $\beta$ that appears in \eqref{WBEA} is held fixed and has been assumed to be bounded, we will set $\norm{\beta}_{\infty} = 1$. We will also take 
$\norm{\s}_{\infty} = 1$. Besides, we take the constant $C$ that appears in the estimates \eqref{B1} - \eqref{growth} to be greater than $1$ in order to avoid writing $C \vee 1$ in many of the estimates in this paper. A generic constant will be denoted by $K$ though it may vary from line to line.

 \smallskip
 Let us fix a finite time $T > 0$, and denote the Skorohod space $\mathbb{D}([0, T]; \mathbb{R}^3)$  by $\mathbb{D}$. We consider it here  equipped with the Skorohod topology. 
 Given a probability measure $\mu$ on $\mathbb{D} \times \mathbb{D}$, let $\mu_t$ denote its marginal at time $t$.  We define 
 \[
 \hat{\alpha}(z,v, \xi) := \alpha(z,v, \xi)\s(|z-v|)
 \]
 for all $z, v \in \mathbb{R}^3$ and $\xi \in \Xi$ (the function $\alpha$ was defined in \eqref{kernel}). The main result of this paper is stated below.

\noindent 
\begin {thm} \label{thm-existence-vel}  Suppose that $\s$ is in $C_b^{\infty}(\R)$. Suppose hypothesis {\bf A} hold.  Let $X_0$ and $Z_0$ be $\mathbb{R}^3$- valued random variables with finite second moments. For any fixed $T>0$, there exists a stochastic basis $(\Omega,\mathscr{F},(\mathscr{F}_t)_{t\in [0,T]},\mathrm{P})$,   an adapted process $(X_t,Z_t)_{t\in [0,T]}$ with values on  $\mathbb{D}\times \mathbb{D}$,  and a compensated random measure (crm) $\tilde {N}^\mu$, with  $\mu$ being the law of the stochastic process $(X, Z)$, satisfying a.s. the following stochastic equation for $t \in [0,T]$:
\begin{equation}
 \begin {split}
   Z_t &=  Z_0 + \int_0^t \int_{\mathbb{D}\times \mathbb{D}\times \Xi \times [0,1]} \alpha(Z_s, v_s, \xi)1_{[0,\; \s(|Z_{s} - v_s|)\beta(|X_s - y_s|)]}(r) 
d\tilde {N}^\mu\\& + \int_0^t \int_{\mathbb{D}\times \mathbb{D} \times \Xi} \hat{\alpha}(Z_s, v_s, \xi)\beta(|X_s - y_s|) d\mu Q(d\xi)  ds \label{eq-vel}
 \end {split}
 \end{equation}
 and 
 \begin{equation}
 \label{eq-space}
  X_t = X_0 + \int_0^t Z_sds, 
 \end{equation}
where  $d\tilde {N}^\mu:=$$\tilde {N}^\mu( dy,dv,d\xi,dr,ds)$. For any $t\in [0,T]$,  $X_t$ and $Z_t$  have finite second moments.
\end {thm}
 \begin {proposition} \label{PropEns} Let $\mu$ denote the law of the process $\{X_s,Z_s: 0\le s \le T\}$, solving \eqref{eq-vel}, \eqref{eq-space}. Then $\mu$  solves 
the Enskog equation (\ref{MWBEA}) for any $\psi \in C^2_0(\R^6)$.
\end {proposition}

 \begin {proof}
   As  for any $t\in [0,T]$ $X_t$ and $Z_t$  have finite second moments, and due to the conditions ({\bf A1}) and \eqref{B2} we can  apply the It\^o formula to $(X_s,Z_s)_{s \in \mathbb{R}_+}$. 
In fact let $t$, $\Delta t > 0$, then 
\begin{align}
& \psi(X_{t+ \Delta t},Z_{t+ \Delta t})\nonumber \\
&= \psi(X_t,Z_t) +  \int_t^{t+ \Delta t} (Z_s,  \nabla_x \psi (X_s, Z_s)) ds \nonumber\\
&+ \int_t^{\Delta t} \int_{U_0\times [0,1]} \{\psi (X_s, Z_s + \alpha(Z_s,v_s,\xi)1_{[0,\; \s(|Z_{s} - v_s|)\beta(|X_s - y_s|)]}(r)) -\psi (X_s, Z_s)  \nonumber \\ & -\nabla_z\psi (X_s,Z_s) \alpha(Z_s, v_s, \theta, \phi)1_{[0,\; \s(|Z_{s} - v_s|)\beta(|X_s - y_s|)]}(r)\} \mu(dy,dv)  Q(d\theta) d\phi ds dr\, +\nonumber\\
&\int_t^{\Delta t} \int_{U_0}    \alpha(Z_s, v_s, \theta, \phi), \nabla_z\psi (X_s,Z_s))\}\s(|Z_s - v_s|)\beta(|X_s - y_s|) 
\mu(dy, dv) Q(d\theta) d\phi ds
\nonumber\\
& + M^{t+ \Delta t}_t(\psi) )\nonumber \\
&= \psi(X_t,Z_t) +  \int_t^{t+ \Delta t} (Z_s,  \nabla_x \psi (X_s, Z_s)) ds \nonumber\\
&+ \int_t^{\Delta t} \int_{U_0} \{\psi (X_s, Z_s + \alpha(Z_s,v_s,\xi)) -\psi (X_s, Z_s)  \nonumber \\ & -\nabla_z\psi (X_s,Z_s) \alpha(Z_s, v_s, \theta, \phi)\} \s(|Z_{s} - v_s|)\beta(|X_s - y_s|) \mu(dy,dv) Q(d\theta) d\phi ds dr \,+\nonumber\\
&\int_t^{\Delta t} \int_{U_0}    \alpha(Z_s, v_s, \theta, \phi), \nabla_z\psi (X_s,Z_s))\}\s(|Z_s - v_s|)\beta(|X_s - y_s|) 
\mu(dy, dv) Q(d\theta) d\phi ds
\nonumber\\
& + M^{t+ \Delta t}_t(\psi)\label{eqIto}
\end{align}
where $\{M^{t+ \Delta t}_t(\psi)\}_{\Delta t \in (0,T]}\}$ is a martingale, for each $T\in \mathbb{R}$.

Taking the expectation $\mathbb{E}$ with respect to the measure $\mu(dx,dv)$ in \eqref{eqIto}, we get 
\begin{align*}
 &\mathbb{E}[\psi(X_{t+ \Delta t},Z_{t+ \Delta t})- \psi(X_t,Z_t)] \\&= \mathbb{E}[ \int_t^{t+ \Delta t} (Z_s, \nabla_x\psi(X_s,Z_s)) ds]\\ & + 
\mathbb{E}[\int_0^t \int_{U_0} 
\s(|Z_s - v_s|)\beta(|X_s - y_s|) (\alpha(Z_s, v_s, \theta, \phi), \nabla_z \psi(X_s,Z_s))
\mu(dy, dv) Q(d\theta) d\phi ds ]\\
&+ \mathbb{E}[\int_{0}^{t}\int_{U_0} \{\psi(X_s, Z_s+ \alpha(Z_s, v_s, \theta, \phi)) - \psi(X_s,Z_s) \notag\\& 
-  (\nabla_z\psi(X_s,Z_s),   \alpha(Z_s, v_s, \theta, \phi))\}\s(|Z_s - v_s|) \beta(|X_s - y_s|)\mu(dy, dv) Q(d\theta) d\phi ds] 
\end{align*}

Dividing by $\Delta t$ and letting $\Delta t \to 0$ we obtain \eqref{MWBEA} by noting that $\mu$ is also the law of $(Z, X)$.
\end {proof}
 In Section 7 in Theorem \ref{ThmUniq} we will prove uniqueness of the  law $\mu$  of the process $\{X_s,Z_s: 0\le s \le T\}$, solving the McKean -Vlasov equation  \eqref{eq-vel}, \eqref{eq-space} in the following sense: 
we prove that for any fixed $t$ in the interval $[0, T]$, the $t$-marginal distribution of weak solutions of \eqref{eq-vel}, \eqref{eq-space} is unique within the class of 
Borel probability measures on $\mathbb{R}^6$ that are absolutely continuous with respect to the Lebesgue measure on $\R^6$. 

\par  The existence of a probability density for the time-marginals of  the velocity is verified  in the case where the initial condition is  Gaussian, and is shown to be the density of an invariant measure in Section 8, Theorem \ref{invsol}.

\section{Existence and uniqueness of  a stochastic equation}
Consider a given  filtered probability space $(\Omega,\mathscr{F},(\mathscr{F}_t)_{t\in [0,T]},\mathrm{P})$ satisfying the usual conditions.   Let $S_T : = S_T^1(\mathbb{R}^d)$ denote the linear  space of all adapted c\`adl\`ag processes $(X_t)_{t\in [0,T]}$ with values on $\mathbb{R}^d$ equipped with norm 
\begin{equation} \label{norm1}
\|X\|_{S_T^1}:= \mathbb{E}[\sup_{s\in[0,T]}|X_s|].
\end{equation}
$S_T^1(\R^d)$ is a Banach space. This can be shown  similar to the proof of Lemma 4.2.1, page 93 in \cite {MR}.
We consider $S_T(\mathbb{R}^d)$ for $d = 6$.  

\par
Let $\mathbb{D}$ be equipped with the Skorohod topology, where $\mathbb{D}$ denotes $\mathbb{D}([0, T]; \R^3)$. 
Given a probability measure $\l$ on $\mathbb{D} \times \mathbb{D}$, let $\l_t$ denote its marginal at time $t$.  Let us assume 

\noindent 
\begin{equation}\int_0^T \int_{\mathbb{D}\times \mathbb{D}}  (|v_t|+|y_t|)\l(dv,dy)dt\,  <\infty \quad \forall \, T>0. \label{L1 condition}
\end{equation} 
Consider a Poisson random measure  $N^{\l}$ on $(\Omega,\mathscr{F},\mathscr{F}_t,\mathrm{P})$ with 
intensity measure $\l_t(dy, dv)Q(d\xi)drdt$ on the Borel subsets of $\mathbb{D} \times \mathbb{D} \times (0, \pi]\times [0, 2\pi)\times [0,1]\times [0, T]$.
We denote by $\tilde{N}^{\l}$ the corresponding compensated Poisson random measure. From the condition (\ref{L1 condition}) and (\ref{growth}), and the hypothesis  $\|\sigma\|_\infty$$= \|\beta\|_\infty=1$, it follows that 
\begin{equation}
\int_0^T\int_{\mathbb{D} \times \mathbb{D} \times \Xi}|\alpha(z,v_t,\xi)|\sigma(|z-v_t|)\beta(|x-y_t|)\l(dv,dy)dt Q(d\xi)\, <\infty \,, 
\end{equation}
$\forall z \in \mathbb{R}^3, x\in \mathbb{R}^3, \forall\, T>0$.

Let us use the following notation: 
\[
U_0 = \mathbb{D}^2 \times \Xi\]
\[
U = \mathbb{D}^2 \times \Xi \times [0, 1] \]
where $\mathbb{D}^2:=\mathbb{D}\times \mathbb{D}$.

\begin{thm}\label{Theorem MR} 
Let  $(Z_0,X_0)$ be a random vector with values on $ \mathbb{R}^3$$\times \mathbb{R}^3$ with 
 \begin{equation}\mathbb{E}[|Z_0|]<\infty,\quad  
\mathbb{E}[|X_0|]<\infty \,,\label{first moment initial}\end{equation} and assume \eqref{L1 condition}. 
 Then for all $T>0$ there  exists a unique strong solution  of the stochastic  equation 
  \begin {align}
    Z_t^{\l} &= Z_0\notag\\&+ \int_0^t \int_{U}\alpha(Z_{s}^{\l}, v_s, \xi)1_{[0,\; \s(|Z_{s}^{\l} - v_s|)\beta(|X_s^{\l} - y_s|)]}(r)
 d\tilde {N}^{\l}\notag\\& + \int_{U_0} \hat{\alpha}(Z_s^{\l}, v_s, \xi)\beta(|X_s^{\l} - y_s|) \l(dy dv) Q(d\xi)ds \label {eq-velr}
\\X_t^{\l}& = X_0+\int_0^t Z_s^{\l}ds,  \label{eq-spacer}
    \end{align}
    on $S_T^1$, where  $d\tilde {N}^{\l}$ denotes $\tilde {N}^{\l}(dy,dv,d\xi, dr, ds)$.
\end{thm}

We first introduce some  notation and preliminary results. 

\par Let $T>0$ and $(Z,X)_{t\in [0,T]}$ be an adapted process with values in $\mathbb{D}^2$. 
\begin{lemma}\label{stochIntdefined}
Assume $(Z,X)_{t\in [0,T]}\in S_T^1$. Then the stochastic integrals $(I(Z))_{t\in [0,T]}$ and $(\hat {I}(Z))_{t\in [0,T]}$, with 
\begin{equation}
I(Z)_{t}:= 
\int_0^t  \int_{U}\alpha(Z_{s_-}, v_s, \xi)1_{[0,\; \s(|Z_{s_-} - v_s|)\beta(|X_{s} - y_s|)]}(r)
 d{N}^{\l}, \label{stochint}
\end{equation}
and 
\begin{equation}\label{stochintmod}
\hat{I}(Z)_{t}:= \int_0^t  \int_{U} |\alpha(Z_{s_-}, v_s, \xi)|1_{[0,\; \s(|Z_{s_-} - v_s|)\beta(|X_{s} - y_s|)]}(r)
 d{N}^{\l}, 
\end{equation}
are well defined, and there exist constants $K>0$ and $M_T>0$ satisfying
\begin{equation}\label{ineqstochintmod}
\mathbb{E}[\hat{I}(Z))_{T}] \leq K \int_0^T \mathbb{E} [\sup_{s\in [0,t]} |Z_s|] dt +M_T. 
\end{equation}
\end{lemma}
\begin{proof}
We need to prove only inequality \eqref{ineqstochintmod}. It then follows that the stochastic integrals $(\hat {I}(Z)))_{t\in [0,T]}$ and $(I(Z)))_{t\in [0,T]}$ are well defined (see Section 3.5, in particular Lemma 3.5.20 of \cite{MR}, or Theorem 4.12 \cite{Ru}). 

\noindent 
Using \eqref{growth} it follows 
\begin{align*}
&\mathbb{E}[\hat{I}(Z)_{T}] =  \int_0^t \int_{U_0 } \mathbb{E}[  |\alpha(Z_{s}, v_s, \xi)| \s(|Z_{s_{-}} - v_s|)\beta(|X^{\l}_{s} - y_s|)]\l(dy dv) Q(d\xi)ds \\& \leq C 
  \left( \int_0^T \mathbb{E} [ |Z_s|] ds + \int_0^T \int_{\mathbb{D}\times \mathbb{D}} |v_s|\l(dv,dy)ds\right) 
  \\& \leq C 
    \left( \int_0^T \mathbb{E} [\sup_{s\in [0,t]} |Z_s|] dt + \int_0^T \int_{\mathbb{D}\times \mathbb{D}} |v_s|\l(dv,dy)ds\right) \,<\infty 
\end{align*}
\end{proof}

Let $ (SZ,SX)_{t\in [0,T]}$ denote the process defined through 
\begin{equation*}
    SZ_t := Z_0 +I(Z)_t \;\;\;\text{and}\;\;\; 
    SX_t :=X_0+\int_0^t Z_s ds.
    \end{equation*}
    If the random vector $(Z_0,X_0)$ satisfies \eqref{first moment initial} and $(Z,X)\in S_T^1$ then $(SZ,SX) \in S_T^1$. 
    
    Indeed 
    \begin{equation}\label{supS}
    \sup_{t\in [0,T]} |SZ_t| \leq \hat{I}(Z)_T +|Z_0| \quad a.s.
    \end{equation}
    and hence 
     \begin{equation}\label{expsupS}
        \mathbb{E}[ \sup_{t\in [0,T]} |SZ_t| ]\leq   \mathbb{E}[\hat{I}(Z)_T] +\mathbb{E}[|Z_0|] 
         \end{equation}
    The statement follows from the estimate  \eqref{ineqstochintmod}.
   
   \begin{lemma}\label{locLip}
    For any $T>0$ fixed, there exists a constant $K>0$, such that for all $n\in \mathbb{N}$ and all $(Z,X)\in S_T^1$ satisfying  $\sup_{s\in [0,T]} |Z_s|\leq n$, $\sup_{s\in [0,T]} |Z_s'|\leq n$, the following inequality holds 
    \begin{align*}
    \int_0^t  \int_{U} & |\alpha(Z_s, v_s, \xi)1_{[0,\; \s(|Z_s - v_s|)\beta([|X_s - y_s|]}(r)- \alpha(Z_s', v_s, \xi)1_{[0,\; \s(|Z_s' - v_s|)\beta(|X_s' - y_s|)]}(r)| \\& \l(dv,dy) Q(d\xi) dr ds 
   \\&\leq  \int_0^t  L_n (|Z_s-Z_s'|+|X_s-X_s'|)ds \quad P\,-a.s., 
    \end{align*}
    with $L_n=Kn$.
    \end{lemma}
    \begin{proof}
    \begin{align*}
        \int_0^t  \int_{U} & |\alpha(Z_s, v_s, \xi)1_{[0,\; \s(|Z_s - v_s|)\beta([|X_s - y_s|]}(r)- \alpha(Z_s', v_s, \xi)1_{[0,\; \s(|Z_s' - v_s|)\beta(|X_s' - y_s|)]}(r)| \\& \l(dv,dy) Q(d\xi) dr ds \leq I+II \quad P\,-a.s., \,\,\,\,\text{with}
        \end{align*}
          \begin{align*}
          I:= &\int_0^t  \int_{U}|\alpha(Z_s, v_s, \xi)-\alpha(Z_s', v_s, \xi)|1_{[0,\; \s(|Z_s - v_s|)\beta([|X_s - y_s|]}(r)\l(dv,dy) Q(d\xi) dr ds\\& \leq  \int_0^t C|Z_s-Z_s'|ds
          \end{align*}
          where the last  inequality follows from \eqref{B3}, and 
\begin{align*}
          II& := \int_0^t  \int_{U}  |\alpha(Z_s', v_s, \xi)|\times \\&
          \{1_{[0,\; \s(|Z_s - v_s|)\beta(|X_s - y_s|)]}(r)-1_{[0,\; \s(|Z_s' - v_s|)\beta(|X_s' - y_s|)]}(r)|\}\l(dv,dy) Q(d\xi) dr ds
          \\& \leq \int_0^t  \int_{U_0} |\alpha(Z_s', v_s, \xi)|| \max(\s(|Z_s - v_s|)\beta(|X_s - y_s|),\s(|Z_s' - v_s|)\beta(|X_s' - y_s|))\\& - \min(\s(|Z_s - v_s|)\beta(|X_s - y_s|),\s(|Z_s' - v_s|)\beta(|X_s' - y_s|)|) \l(dv,dy) Q(d\xi) ds =
          \\& \int_0^t  \int_{U_0} |\alpha(Z_s', v_s, \xi)|\, | \s(|Z_s - v_s|)\beta(|X_s - y_s|)- \s(|Z_s' - v_s|)\beta(|X_s' - y_s|)| \l(dv,dy) Q(d\xi) ds
          \end{align*}
          Using that $\sigma$ and $\beta$ are Lipschitz continuous functions bounded by $1$, as well as \eqref{growth}, we get that there exists a constant $K>0$, such that 
           \begin{align*}
                     II&\leq \int_0^t  \int_{U_0} |\alpha(Z_s', v_s, \xi)|\\& \times (||Z_s'-v_s|-|Z_s-v_s|| + ||X_s'-y_s|-|X_s-y_s||) \l(dv,dy) Q(d\xi) ds
                    \\&\leq K \int_0^t  \int_{U_0} |\alpha(Z_s', v_s, \xi)| (|Z_s'-Z_s| +|X_s'-X_s|) \l(dv,dy) Q(d\xi) ds
                     \\ &\leq K \int_0^t  \int_{\mathbb{D}^2} (|Z_s'|+|v_s|)| (|Z_s'-Z_s| +|X_s'-X_s|) \l(dv,dy) ds 
            \end{align*}
            \end{proof}
In the next Lemma we will use the local Lipschitz condition stated in Lemma \ref{locLip} to prove a local contraction property of $S$ on $S_T^1$.
\begin{lemma}
\label{contraction}
For each $n\in \mathbb{N}$ there exists  constant $L_n>0$, such that
\begin{align*}
&\mathbb{E}[\sup_{s\in [0,t]} |SZ_s-SZ_s'|] \leq  L_n \int_0^t \mathbb{E}[\sup_{s'\in [0,s]} \{|Z_{s'}-Z_{s'}'|+|X_{s'}-X_{s'}'|\}]ds \\
&\quad \forall (Z,X)_{s\in [0,T]}\,, (Z',X')_{s\in [0,T]} \, \in S_T^1,\, \text{with} \,  \sup_{s\in [0,T]}|Z_s|\leq n\,, \sup_{s\in [0,T]}|Z'_s|\leq n. \label{eqcontraction}
\end{align*}
\end{lemma}
\begin{proof}
\begin{align*}
&\mathbb{E}[\sup_{s\in [0,t]} |SZ_s-SZ_s'| ]\leq \\&
\mathbb{E}[\sup_{s\in [0,t]} \int_0^s  \int_{U}  |\alpha(Z_{s'_-}, v_{s'}, \xi)
1_{[0,\; \s(|Z_{s'_-} - v_{s'}|)\beta(|X_{s'} - y_{s'}|)]}(r) \\&-\alpha(Z'_{s'_-}, v_{s'}, \xi)
1_{[0,\; \s(|Z'_{s'_-} - v_{s'}|)\beta(|X'_{s'} - y_{s'}|)]}(r)| dN_\l ]  \\& 
\leq  
\mathbb{E}[ \int_0^t  \int_{U}   |\alpha(Z_{s'_-}, v_{s'}, \xi)
1_{[0,\; \s(|Z_{s'_-} - v_{s'}|)\beta(|X_{s'} - y_{s'}|)]}(r) \\&  -\alpha(Z'_{s'_-}, v_{s'}, \xi)
1_{[0,\; \s(|Z'_{s'_-} - v_{s'}|)\beta(|X'_{s'} - y_{s'}|)]}(r)| dN_\l ]
\\&=\mathbb{E}[ \int_0^t  \int_{U}   |\alpha(Z_{s'}, v_{s'}, \xi)
1_{[0,\; \s(|Z_{s'} - v_{s'}|)\beta(|X_{s'} - y_{s'}|)]}(r) \\& -\alpha(Z'_{s'}, v_{s'}, \xi)
1_{[0,\; \s(|Z'_{s'} - v_{s'}|)\beta(|X'_{s'} - y_{s'}|)]}(r)|\l(dy dv) Q(d\xi) dr ds']
\end{align*}
\begin{align*}
& \leq L_n \int_0^t \mathbb{E}[ |Z'_{s'}-Z_{s'}|+ |X'_{s'}-X_{s'}| ]ds'
\\&\leq L_n \int_0^t \mathbb{E}[\sup_{s'\in [0,s]} \{|Z'_{s'}-Z_{s'}|+ |X'_{s'}-X_{s'}|\}] ds
\end{align*}
where  we have used Lemma \ref{locLip}.
\end{proof}
In the proof of the next theorem we will use the local contraction property in Lemma \ref{contraction} to prove existence and uniqueness of a modification of the stochastic equation defined through \eqref{eq-velr}, \eqref{eq-spacer}. The modified stochastic equation satisfies global growth and  Lipschitz conditions.

\par
Let $j\in \mathbb{N}$, $B_j:=\{z\in \mathbb{R}^3:\, |z|\leq j\}$ and 
\begin{equation} \label{alpha-loc}
\alpha_j(z,v,\xi):=\frac{\alpha(z,v,\xi)}{1+d(z,B_j)}
\end{equation}
where $d(z,B_j)$ denotes the distance of $z\in \mathbb{R}^3$ from $B_j$. 

\begin{thm}\label{existence-uniqueness Lipschitz}
Let the  random vector  $(Z_0,X_0)$ satisfy \eqref{first moment initial}. 
 For all $T>0$ there  exists a unique solution  on $S_T^1$  of the stochastic  equation    
  \begin {align}
    Z_t^{\l,j} &= Z_0+ \int_0^t \int_{U}\alpha_j(Z_{s_-}^{\l,j}, v_s, \xi)\notag\\
    &\times 1_{[0,\; \s(|Z_{s_-}^{\l,j} - v_s|)\beta(|X_s^{\l,j} - y_s|)]}(r)
 {N}^{\l}(dy,dv,d\xi, dr, ds) \label{eq-velr-j}\\
X_t^{\l,j}& = X_0+\int_0^t Z_s^{\l,j}ds.  \label{eq-spacer-j}
    \end{align}
\end{thm}
It then follows directly  the statement of the following corollary:
\begin{corollary}\label{existence-uniqueness Lipschitz compensated}
Let the random vector  $(Z_0,X_0)$ satisfy \eqref{first moment initial}. 
 For all $T>0$ there  exists a unique solution  on $S_T^1$ of the stochastic  equation    
  \begin {align}
   & Z_t^{\l,j} = Z_0\notag\\& +\int_0^t \int_{U}\alpha_j(Z_{s}^{\l,j}, v_s, \xi)1_{[0,\; \s(|Z_{s}^{\l,j} - v_s|)\beta(|X_s^{\l,j} - y_s|)]}(r)
 d\tilde {N}^{\l}\notag\\& + \int_{U_0} \alpha_j(Z_s^{\l,j}, v_s, \xi) \sigma(|Z_s^{\l,j} - v_s|)\beta(|X_s^{\l,j} - y_s|) d\l(dy dv) Q(d\xi)ds \label {eq-velr-j-com}
\\X_t^{\l,j}& = X_0+\int_0^t Z_s^{\l,j}ds.  \label{eq-spacer-j-com}
    \end{align}
\end{corollary}
\begin{proof} If the stochastic integrals in  the stochastic equation  \eqref{eq-velr-j}, \eqref{eq-spacer-j} are well-defined then  the stochastic equation  \eqref{eq-velr-j}, \eqref{eq-spacer-j} is equivalent to  the stochastic equation  \eqref{eq-velr-j-com}, \eqref{eq-spacer-j-com}. (See e.g. Chapter 5 \cite{MR}).
\end{proof}

\bf{ Proof of Theorem \ref{existence-uniqueness Lipschitz}}
\begin{proof} We start by remarking that 
\begin{equation*}%\label{growth normalized}
\frac{|z|}{1+d(z,B_j)}\leq \min(j,|z|)
\end{equation*}
and there exists a constant $K_j>0$, such that
\begin{equation}\label{Lipschitz normalized}
\left|\frac{z}{1+d(z,B_j)}-\frac{z'}{1+d(z',B_j)}\right| \leq K_j|z-z'|.
\end{equation}
Assume   $(Z,X)_{t\in [0,T]}\in S_T^1$. As $|\alpha_j(z,v,\xi)|\leq |\alpha(z,v,\xi)|$ it follows from Lemma \ref{stochIntdefined} that  the stochastic integrals $(I(Z))_{t\in [0,T]}$ and $(\hat {I}(Z))_{t\in [0,T]}$, with 
\begin{equation*}
\hat{I_j}(Z)_{t} := \int_0^t  \int_{U} |\alpha_j(Z_{s_-}, v_s, \xi)|1_{[0,\; \s(|Z_{s_-} - v_s|)\beta(|X_{s} - y_s|)]}(r)
 d{N}^{\l},%\label{stochintmodj} 
\end{equation*} and 
\begin{equation*}
I_j(Z))_{t}:= \int_0^t  \int_{U}\alpha_j(Z_{s_-}, v_s, \xi)1_{[0,\; \s(|Z_{s_-} - v_s|)\beta(|X_{s} - y_s|)]}(r)
 d{N}^{\l}, %\label{stochintj}
\end{equation*}
are well defined. Moreover, using  $(\hat{I}_j(Z))_{t}\leq (\hat{I}(Z))_{t} $  $\forall t\in [0,T]$ and \eqref{ineqstochintmod}, it follows  that 
\begin{equation}\label{ineqstochintmodj}
\mathbb{E}[\hat{I_j}(Z))_{T}] \leq K \int_0^T \mathbb{E} [\sup_{s\in [0,t]} |Z_s|] dt +M_T\,, 
\end{equation}
with $K>0$, $M_T>0$. 
Let $ (S_jZ,S_jX)_{t\in [0,T]}$ denote the process defined through 
\begin{equation*}
    S_jZ_t := Z_0 +I_j(Z)_t \;\;\text{and}\;\;\; 
    S_jX_t :=X_0+\int_0^t Z_s ds\,.   
    \end{equation*}
    If the random vector $(Z_0,X_0)$ satisfies \eqref{first moment initial},  then 
    \begin{equation}\label{expsupSj}
            \mathbb{E}[ \sup_{t\in [0,T]} |S_jZ_t| ]\leq   \mathbb{E}[\hat{I_j}(Z)_T] +\mathbb{E}[|Z_0|] 
             \end{equation}
             and, due to the growth condition \eqref{ineqstochintmodj}, $(S_jZ,S_jX)_{t\in [0,T]}\in S_T^1$. 
    
    From \eqref{Lipschitz normalized} and Lemma \ref{locLip} it follows that there is a constant $L_j>0$ such that  
    \begin{align*}
        \int_0^t  \int_{U} & |\alpha_j(z, v_s, \xi)1_{[0,\; \s(|z - v_s|)\beta([|x - y_t|]}(r)- \alpha_j(z', v_s, \xi)1_{[0,\; \s(|z' - v_s|)\beta(|x - y_s|)]}(r)| \\& \l(dv,dy) Q(d\xi) dr ds 
       \\&\leq L_j (|z-z'|+|x-x'|)\quad \text{for all }\,  z,z',x,x'\in \mathbb{R}^3, 
        \end{align*}
        Similar to Lemma \ref{contraction} it can then be proven that the following inequality holds:
        \begin{equation*}\mathbb{E}[\sup_{s\in [0,t]} |S_jZ_s-S_jZ_s'|] \leq  L_j \int_0^t \mathbb{E}[\sup_{s'\in [0,s]} \{|Z_{s'}-Z_{s'}'|+|X_{s'}-X_{s'}'|\}]ds\end{equation*}
        It follows that there exists $n\in \mathbb{N}$ such that  $(S^n_jZ,S^n_jX)_{t\in [0,T]}$ is a contraction from $S_T^1$ to $S_T^1$. It then follows,  that the mapping $S_j$ has a unique fixed point on $S_T^1$.
\end{proof}
{\bf Proof of Theorem \ref{Theorem MR}:}
\begin{proof}
It is sufficient to prove that for all $T>0$ there exists a unique process  $ (Z,X)_{t\in [0,T]}$$\in S_T^1$ satisfying a.s. the following stochastic equation 
\begin {align}
   & Z_t^{\l} = Z_0+ \notag \\&\int_0^t \int_{U}\alpha(Z_{s}^{\l}, v_s, \xi)1_{[0,\; \s(|Z_{s}^{\l} - v_s|)\beta(|X_s^{\l} - y_s|)]}(r)
  {N}^{\l}(dy,dv,d\xi, dr, ds) \label{eq-velr-withoutcom}
\\& X_t^{\l} = X_0+\int_0^t Z_s^{\l}ds,  \label{eq-spacer-withoutcom}
    \end{align}
    \eqref{eq-velr-withoutcom}, \eqref{eq-spacer-withoutcom} is then equivalent to  \eqref{eq-velr}, \eqref{eq-spacer}.

  We will follow the strategy of the proof of Theorem 4.3.1 in \cite{MR}.   
Let $(Z^{\l,j},X^{\l,j})_{t\in [0,T]}$$ \in S_T^1$ be the unique solution of  \eqref{eq-velr-j}, \eqref{eq-spacer-j}. Let
\begin{equation*}
\tau_j:=\inf\{t\in [0,T]:|Z_t^{\l,j}|>j\}
\end{equation*}
By uniqueness of the solution of \eqref{eq-velr-j}, \eqref{eq-spacer-j} it follows that 
\begin{equation*}
Z_t^{\l,j}=Z_t^{\l,j+1}\quad \text{a.s.}, \,\,\,\text{for}\,\,\, t\in [0,\tau_j]\,,
\end{equation*}
giving $\mathrm{P}(\tau_j\leq \tau_{j+1})=1$ $\forall j\in \mathbb{N}$ .

 We will prove 
\begin{equation} \label{fulltime}
\mathrm{P}( \cup_{j\in \mathbb{N}}\{\tau_j=T\})=1.
\end{equation}
It then follows that the a.s. limit process $(Z^\l,X^\l)_{t\in [0,T]}$$=\lim_{j\to \infty}$$(Z^{\l,j},X^{\l,j})_{t\in [0,T]}$  is the solution of  \eqref{eq-velr-withoutcom}, \eqref{eq-spacer-withoutcom}, and hence   \eqref{eq-velr}, \eqref{eq-spacer}.

It follows from \eqref{expsupSj} and  \eqref{ineqstochintmodj} that 
\begin{equation}\label{expsupZj}
            \mathbb{E}[ \sup_{t\in [0,T]} |Z^{\l,j}_t| ]\leq K \int_0^T \mathbb{E} [\sup_{s\in [0,t]} |Z^{\l,j}_s|] dt +M_T  +\mathbb{E}[|Z_0|] 
             \end{equation}
             so that by Gronwall's Lemma 
             \begin{equation}\label{GrownwallZj}
                         \mathbb{E}[ \sup_{t\in [0,T]} |Z^{\l,j}_t| ]\leq \exp{KT} (M_T  +\mathbb{E}[|Z_0|]) 
                          \end{equation}
     It follows 
     \begin{align*} \mathrm{P}(\tau_j<T)&=\mathrm{P}(\sup_{t\in [0,T]} |Z^{\l,j}_t|>j)\\& \leq \frac{1}{j}\mathbb{E}[ \sup_{t\in [0,T]} |Z^{\l,j}_t| ] \leq \frac{1}{j}\exp{KT} (M_T  +\mathbb{E}[|Z_0|]) \end{align*}
     so that 
     \begin{equation} \label{fulltimecomplement}
     \mathrm{P}( \cap_{j\in \mathbb{N}}\{\tau_j<T\})=\lim_{j\to\infty} \mathrm{P}(\{\tau_j< T)\}=0.
     \end{equation}            
\end{proof}
\section{Tightness} 
\rm{In this section, we formulate  an approximating sequence $\{Z^{(n)}, X^{(n)})$ for the McKean-Vlasov limit, and prove the tightness of this sequence by Kurtz's criterion on the Skorohod space $\mathbb{D}\times \mathbb{D}$ with the Skorohod topology.}

\noindent 
Define the processes
\begin{align*}
Z_t^{(0)} & = Z_0\\
X_t^{(0)} &= X_0 + Z_0 t
\end{align*}
for all $t \in [0, T]$.
Let $\mu^{(0)} := \mathcal{L} (Z^{(0)}, X^{(0)})$.  Consider a Poisson random measure $N_{\mu^{(0)} }$ on $U\times [0, T]$ whose compensator measure is given by $d\mu^{(0)} Q(d\xi) dr ds$. 
Let $\tilde{N}_{\mu^{(0)} }$ denote the corresponding compensated Poisson random measure (cPrm).
By the square integrability of $Z_0$ and $X_0$, one has 
\[
\mathbb{E} (\sup_{0 \le t \le T} \left[|Z_t^{(0)}|^2 + |X_t^{(0)}|^2\right] < \infty\,. 
\]
This implies in particular that $\forall T>0$ $(Z^{(0)}, X^{(0)})$ $\in $ $S_T^2\subset $ $S_T^1$, where  $S_T^2$  denotes here  the Banach  space of all adapted c\`adl\`ag processes $(X_t)_{t\in [0,T]}$ with values on $\mathbb{R}^6$ equipped with norm 
\begin{equation} \label{norm}
\|X\|_{S_T^2}:= (\mathbb{E}[\sup_{s\in[0,T]}|X_s|^2])^{1/2}
\end{equation}
(see e.g. page 93 in \cite {MR}). 
In particular it implies that the measure $\mu^{(0)}$ is square integrable, and as a consequence satisfies the assumption \eqref{L1 condition}.  For all $n \ge 0$, define 
\begin{align}
Z_t^{(n+1)} &= Z_0 + \int_0^t \int_U \alpha(Z_s^{(n+1)}, v_s, \xi) 1_{[0, \s(|Z_s^{(n+1)} -v_s|)\beta(|X_s^{(n+1)} - y_s|)]} (r) d\tilde{N}_{\mu^{(n)} } \notag\\
& + \int_0^t \int_{U_0} \hat{\a}(Z_s^{(n+1)}, v_s, \xi) \beta(|X_s^{(n+1)} - y_s|)d\mu^{(n)}Q(d\xi) ds \label{eqnZn}
\end{align}
and 
\begin{equation}
X_t^{(n+1)} = X_0 + \int_0^t Z_s^{(n+1)}ds. \label{eqnXn}
\end{equation}
Here, $\mu^{(n)}$ is the law of $(Z^{(n)}, X^{(n)})$,  and $\tilde{N}_{\mu^{(n)} }$ is the cPrm with compensator measure  given by $d\mu^{(n)} Q(d\xi) dr ds$. 
 Taking $n =0$, it follows from Theorem \ref{Theorem MR} that there exists a unique solution of  $(Z^{(1)}, X^{(1)})$ in  $S_T^1$ solving  \eqref{eqnZn}, \eqref{eqnXn}.  Moreover,  
\begin{align*}
&\mathbb{E}(|Z_t^{(1)}|^2) \\
& \le 3  \mathbb{E}\left[(|Z_0|^2) + \int_0^t \int_{U_0} |\a(Z_s^{(1)}, v_s, \xi)|^2 \s(|Z_s^{(1)} -v_s|)\beta(|X_s^{(1)} - y_s|)d\mu^{(0)}Q(d\xi) ds\right.\\
& \left. + | \int_0^t \int_{U_0} \hat{\a}(Z_s^{(1)}, v_s, \xi) \beta(|X_s^{(1)} - y_s|)d\mu^{(0)}Q(d\xi) ds|^2\right]\\
&\le 3  \mathbb{E}\left[(|Z_0|^2) + \int_0^t \int_{U_0} |\a(Z_s^{(1)}, v_s, \xi)|^2 d\mu^{(0)}Q(d\xi) ds \,\;+ \right.\\
& \left. t\int_0^t \int_{\mathbb{D}^2} | \int_{\Xi}\a(Z_s^{(1)}, v_s, \xi) Q(d\xi)|^2 d\mu^{(0)} \int_{\mathbb{D}^2}|\s(|Z_s^{(1)} -v_s|)\beta(|X_s^{(1)} - y_s|)|^2d\mu^{(0)}ds\right]\\
\end{align*}
by Cauchy-Schwarz inequality; continuing, by a use of \eqref{B1}, \eqref{B3}, and \eqref{B2}, one obtains
\begin{align*}
&\le 3  \mathbb{E}\left[(|Z_0|^2) + C\int_0^t [|Z_s^{(1)}|^2 + |v_s|^2] d\mu^{(0)} ds + Ct\int_0^t [|Z_s^{(1)}|^2 + |v_s|^2]d\mu^{(0)} ds\right]\\
& \le 3\left[ \mathbb{E}|Z_0|^2 + C(1+T) \int_0^t \mathbb{E} |Z_s^{(1)}|^2ds + Ct(1 + T)\mathbb{E}|Z_0|^2\right]
\end{align*}
Hence by the Gronwall inequality, we obtain
\begin{equation}
\mathbb{E}(|Z_t^{(1)}|^2)  \le K_1 E(|Z_0|^2) (1 + K_2t) \label{gr1}
\end{equation}
with $K_1 = 3e^{3CT(1 + T)}$ and $K_2 = C(1 +T)$. It then follows  for $n=1$ that the law  $\mu^{(1)}$ satisfies the assumption \eqref{L1 condition}, so that  there exists a unique strong solution   $(Z^{(2)}, X^{(2)})$  solving \eqref{eqnZn}, \eqref{eqnXn}.   Along similar lines, one obtains
\begin{align*}
&\mathbb{E}(|Z_t^{(2)}|^2) \\
& \le 3 \mathbb{E}\left[|Z_0|^2 + C\int_0^t ( |Z_s^{(2)}|^2  + (|Z_s^{(1)}|^2))ds + CT\int_0^t (|Z_s^{(2)}|^2  + (|Z_s^{(1)}|^2))ds\right]
\end{align*}
so that by the Gronwall inequality, 
\begin{equation}\label{gr2}
\mathbb{E} |Z_t^{(2)}|^2 \le K_1E(|Z_0|^2) (1 + K_1K_2t + \frac{(K_1K_2t)^2}{2})
\end{equation}
Further iterations result that \eqref{eqnZn}, \eqref{eqnXn} has a unique strong solution  $(Z^{(n)}, X^{(n)})$  and  the bound
\[
\mathbb{E} |Z_t^{(n)}|^2 \le K_1E(|Z_0|^2) \sum_{i =0}^n\frac{(K_1K_2t)^i}{i!}\,,
\]
so that for all $n \in \mathbb{N}$, we have
\begin{equation}\label{t1}
\mathbb{E} |Z_t^{(n)}|^2 \le K_1 e^{K_1 K_2 t} \mathbb{E}|Z_0|^2.
\end{equation}
By the definition of $X^{(n)}$, we obtain an upper bound uniformly in $n$ for $\mathbb{E} \left[|Z_t^{(n)}|^2 + |X_t^{(n)}|^2\right]$.
This is uniform boundedness of the sequence at each fixed $t \in [0, T]$.  Using the Burkholder-Davis-Gundy inequality, and proceeding exactly as above, 
one obtains an upper bound $K$ uniformly in $n$ for $\mathbb{E}\left[\sup_{0 \le t \le T}( |Z_t^{(n)}|^2 + |X_t^{(n)}|^2 ) \right]$. It follows in particular that  $(Z^{(n)}, X^{(n)})$ $\in S_T^2\subset S_T^1$. As $S_T^p$, $p=1,2$ are not separable Banach spaces, tightness has however to be proven on the Skorohod space $\mathbb{D} \times \mathbb{D}$ with the Skorohod topology: 

\sni
in order to verify the second requirement in Kurtz's criterion, we consider for any fixed $\d > 0$, 
\begin{align*}
& \mathbb{E}\left[|Z_{t+\d}^{(n)} - Z_t^{(n)}|^2\,|\,\mathcal{F}_t\right]\\
&\le 2\mathbb{E}\left[ \{|\int_t^{t+\d} \int_U \a(Z_s^{(n)}, v_s, \xi)1_{[0, \s(|Z_s^{(n)} - v_s|)\beta(|X_s^{(n)} - y_s|)]}(r) d\tilde{N}_{\mu^{(n-1)}}|^2\right.\\
&\left.\,\,+ |\int_t^{t+\d} \int_{U_0} \hat{\a}(Z_s^{(n)}, v_s, \xi)\beta(|X_s^{(n)} - y_s|)Q(d\xi)d\mu^{(n-1)}ds|^2\}\,| \mathcal{F}_t \right]\\
& \le 2\mathbb{E}\left[\{|\int_t^{t+\d} \int_{U_0} |\a(Z_s^{(n)}, v_s, \xi)|^2 \s(|Z_s^{(n)} - v_s|)\beta(|X_s^{(n)} - y_s|)d\mu^{(n-1)} Q(d\xi)ds\right.  \\
&  +  \d \int_t^{t + \d}\int_{\mathbb{D}^2}| \int_{\Xi} \a(Z_S^{(n)}, v_s, \xi)Q(d\xi)|^2 d\mu^{(n-1)} \times \\
&\left. \,\,\,\int_{\mathbb{D}^2} \s^2(|Z_s^{(n)} - v_s|)\beta^2(|X_s^{(n)} - y_s|)d\mu^{(n-1)}ds\}\,|\,\mathcal{F}_t\right].
\end{align*}
We will call the above expression on the right side  as $ 2\mathbb{E}(A_{\d}^{(n)}\,|\,\mathcal{F}_t)$. Then,
\begin{align*}
\mathbb{E}(A_{\d}^{(n)}\,|\,\mathcal{F}_t) &\le 2C(1  + \d) \int_t^{t + \d} \mathbb{E} (|Z_s^{(n)}|^2 + |Z_s^{(n-1)}|^2 )ds\\
& \le K \d
\end{align*}
for a suitable constant $K > 0$ which is independent of $n$. 
Hence, 
\begin{equation}\label{t2}
\lim_{\d \to 0} \limsup_{n \to \infty}\mathbb{E} (A_{\d}^{(n)}) = 0.
\end{equation}
From \eqref{t1} and \eqref{t2}, we conclude that $\{Z^{(n)}\}$ is tight in $\mathbb{D}$. By the definition of $X^{(n)}$, it follows that 
$\{Z^{(n)}, X^{(n)})$ is tight in $\mathbb{D}^2$.
%%%%%%%%%%%%%%%%%%%%%%%%%%%%%%
%%%%%%%%%%%%%%%%%%%%%%%%%%%%%%
%%%%%%%%%%%%%%%%%%%%%%%%%%%%%%
%%%%%%%%%%%%%%%%%%%%%%%%%%%%%%
\section{Distance Between Successive Approximations}
In this section, we give a result on the closeness of the measures $\mu^{(n+1)}$ and $\mu^{(n)}$ as $n \to \infty$. 
\begin{proposition}\label{prop5.1}
 Suppose that $\s$ is in $C_b^{\infty}(\R)$. Then for all $h \in C_b^{\infty}(\R^6)$ and  for any fixed $t \in [0, T]$,  we have
 \begin{equation}
 \mathbb{E}[h(Z_t^{(n+1)}, X_t^{(n+1)}) - h(Z_t^{(n)}, X_t^{(n)})] \to 0 
 \end{equation}
 as $n \to \infty$.
 \end{proposition}
 \begin{proof}
 By the It\^o formula, we write $\mathbb{E}[h(Z_t^{(n+1)}, X_t^{(n+1)}) - h(Z_t^{(n)}, X_t^{(n)})]$ as $A_1 + A_2$ where
 \[ A_1 = \mathbb{E}[ \int_0^t \{\nabla_x h(Z_s^{(n + 1)}, X_s^{(n+1)}) \cdot Z_s^{(n+1)} - \nabla_x h(Z_s^{(n)}, X_s^{(n)}) \cdot Z_s^{(n)}\}ds].\]
 and 
 \begin{align*}
 A_2 &= \mathbb{E} [ \int_0^t \int_{U_0} \{h(Z_s^{(n + 1)} + \a(Z_s^{(n + 1)}, v_s, \xi), X_s^{(n+1)}) - h(Z_s^{(n + 1)}, X_s^{(n+1)}) \}\\
 &\,\,\,\,\times  \s(|Z_s^{(n + 1)} - v_s|) \beta(|X_s^{(n+1)} - y_s|)d\mu^{(n)}Q(d\xi)ds]\\
 & - \mathbb{E} [ \int_0^t \int_{U_0} \{h(Z_s^{(n)} + \a(Z_s^{(n)}, v_s, \xi), X_s^{(n)}) - h(Z_s^{(n)}, X_s^{(n)}) \} \\
 & \,\,\,\,\times \s(|Z_s^{(n)} - v_s|) \beta(|X_s^{(n)} - y_s|)d\mu^{(n-1)}Q(d\xi)ds]
 \end{align*}
 By tightness, given any $\e > 0$, there exists an $R > 0$ such that 
 \[
 P\{\sup_{0 \le t \le T}|\max \{|Z_t^{(j)}|  +  |X_t^{(j)}|: j = n-1, n, n+1\} > R\} < \e.
 \]
 This is a statement about the measures $\mu^{(n+1)}, \mu^{(n)}$ and $\mu^{(n-1)}$.  Let $B_R$ denote the $R$-ball in $\R^6$. 
 First, we will deal with $A_1$. Clearly, by Cauchy-Schwarz inequality,  
\begin{align*}
&\vert \int_0^t  \int_{B_R^C} \{\nabla_x h(z_s^{(n + 1)}, x_s^{(n+1)}) \cdot z_s^{(n+1)}\}d\mu^{(n+1)}ds - \int_0^t \int_{B_R^C} \{\nabla_x h(z_s^{(n)}, x_s^{(n)}) \cdot z_s^{(n)}\}
d\mu^{(n)}\}ds \vert \\
&\le \norm{\nabla_x h}_{\infty} \sqrt{ \e t} \left[\int_{B_R^C}\int_0^t \{|z_s^{(n+1)}|^2 d\mu^{(n+1)}ds + |z_s^{(n)}|^2 d\mu^{(n)}\}ds\}\right]^{1/2}\\
& \le K_t \sqrt{\e}
\end{align*} 
for a suitable constant $K_t >  0$.
 Restricted to $B_R$, note that the function $g(z, x) = \nabla_x h(z, x)\cdot z$ can be uniformly approximated by functions in $C_b^{\infty}(\R^6)$, so that 
 \begin{equation}
 A_1 \le K\sqrt{\e} + \int_0^t \sup_{\phi \in C_b^{\infty}(\R^6)} |\int_{\R^6} \phi(z, x) \{d\mu_s^{(n+1)} - d\mu_s^{(n)}\}|ds. \label{A1}
 \end{equation}
In order to bound $A_2$, we will split $A_2$ into two parts so that $A_2 \le I_1 + I_2$ where 
\begin{align*}
I_1 &= \mathbb{E} [ \int_0^t \int_{U_0} \{h(Z_s^{(n + 1)} + \a(Z_s^{(n + 1)}, v_s, \xi), X_s^{(n+1)}) - h(Z_s^{(n + 1)}, X_s^{(n+1)}) \}\\
 &\,\,\,\,\times  \s(|Z_s^{(n + 1)} - v_s|) \beta(|X_s^{(n+1)} - y_s|)d\mu^{(n)}Q(d\xi)ds]\\
& -  \mathbb{E} [ \int_0^t \int_{U_0} \{h(Z_s^{(n)} + \a(Z_s^{(n)}, v_s, \xi), X_s^{(n)}) - h(Z_s^{(n)}, X_s^{(n)}) \}\\
 &\,\,\,\,\times  \s(|Z_s^{(n)} - v_s|) \beta(|X_s^{(n)} - y_s|)d\mu^{(n)}Q(d\xi)ds], 
 \end{align*}
 and
 \begin{align*}
 I_2 &= \mathbb{E} [ \int_0^t \int_{U_0} \{h(Z_s^{(n)} + \a(Z_s^{(n)}, v_s, \xi), X_s^{(n)}) - h(Z_s^{(n)}, X_s^{(n)}) \}\;\;\;\;\;\;\;\;\;\;\;\;\;\,\,\,\,\,\,\,\,\,\\
 &\,\,\,\,\times  \s(|Z_s^{(n)} - v_s|) \beta(|X_s^{(n)} - y_s|)d\mu^{(n)}Q(d\xi)ds]\\
 & - \mathbb{E} [ \int_0^t \int_{U_0} \{h(Z_s^{(n)} + \a(Z_s^{(n)}, v_s, \xi), X_s^{(n)}) - h(Z_s^{(n)}, X_s^{(n)}) \} \\
 & \,\,\,\,\times \s(|Z_s^{(n)} - v_s|) \beta(|X_s^{(n)} - y_s|)d\mu^{(n-1)}Q(d\xi)ds].
 \end{align*}
 To bound $I_1$, we will use the notation $z, x$ instead of $z_s, x_s$ in defining the function 
 \[
 \psi(z, x) = \int_{\R^6\times \Xi} \{h(z+ \a(z, v, \xi), x) - h(z, x)\} \s(|z - v|)\beta(|x - y|)d\mu^{(n)}_s Q(d\xi)
 \]
 where  $s$ is  fixed. Given $\e > 0$, there exists an $R > 0$ such that for all $n$, 
 \[
 \mathbb{E}\int_0^t 1_{\{|Z_s^{(n)}|  > R\}} |Z_s^{(n)}|ds < \e.
 \]
 A similar statement holds for $X^{(n)}$ in the place of $Z^{(n)}$. Therefore, 
 we obtain
 \begin{equation}
 \int_{B_R^C}  |\psi(z, x)| (d\mu_s^{(n)} + \mu_s^{(n+1)}) < K \e \label{oneI1}
 \end{equation}
 where $K$ is a suitable constant, and $B_R$ is the $R$-ball in $\R^6$. 
 
 \sni
 Hence, we can focus our attention on $I_1$ when the processes are restricted to values in $B_R$.
 Next, let $\Xi_{\d}$ denote the subset $(0, \d] \times [0, \pi)$ of $\Xi$.  Given $\e > 0$, there exists a $\d > 0$ such that for all $n$ and $s$, 
 \begin{align}
&|\int_{\R^6} \int_{B_R \times \Xi_{\d}}\{h(z+ \a(z, v, \xi), x) - h(z, x)\} \s(|z - v|)\beta(|x - y|)\mu^{(n)}_s(dv, dy) Q(d\xi)|\notag\\
&\;\;\;\;\;\;\;\;\;\;\;\;(d\mu_s^{(n)} + \mu_s^{(n+1)}) \notag\\
& < \e.\label{twoI1}
\end{align}
With this estimate in hand, we observe that the function of $(z, x)$ given by
\[
\int_{{B_R} \times (\d, 2\pi] \times [0, \pi)}\{h(z+ \a(z, v, \xi), x) - h(z, x)\} \s(|z - v|)\beta(|x - y|)d\mu^{(n)}_s Q(d\xi)
\]
is a function in $C_b^{\infty}(\R^6)$. This observation along with \eqref{oneI1} and \eqref{twoI1} yields 
\begin{equation}
|I_1| \le K\e + \int_0^t \sup_{\phi \in C_b^{\infty}(\R^6)} |\int_{\R^6} \phi(z, x) \{d\mu_s^{(n+1)} - d\mu_s^{(n)}\}|ds.\label{eqnI1}
\end{equation}
For bounding $I_2$, we repeat a procedure similar to the one  used for $I_1$ for the function 
\[
g(v, y) = \int_{\R^6 \times \Xi} \{h(z + \a(z, v, \xi), x) - h(z, x)\} \s(|z - v|) \beta(|x - y|) Q(d\xi) \mu_s^{(n)}(dz, dx)
\]
where $s$ is fixed and the notation $v, y$ is used instead of $v_s, y_s$. 
We obtain the estimate
\begin{equation}
|I_2| \le K\e + \int_0^t \sup_{\phi \in C_b^{\infty}(\R^6)} |\int_{\R^6} \phi(v, y) \{d\mu_s^{(n)} - d\mu_s^{(n-1)}\}|ds \label{eqnI2}
\end{equation}
for a suitable constant $K > 0$. 
Combining \eqref{A1}, \eqref{eqnI1}, \eqref{eqnI2}, we conclude that for  suitable constants $K$, and $C_1$, 
\begin{align*}
 &|\int_{\R^6} h(z, x) \{d\mu_t^{(n+1)} - d\mu_t^{(n)}\}| \\
 &\le K\e + C_1 \int_0^t \sup_{\phi \in C_b^{\infty}(\R^6)} |\int_{\R^6} \phi(z, x) \{d\mu_s^{(n+1)} - d\mu_s^{(n)}\}|ds\\
 & + \int_0^t \sup_{\phi \in C_b^{\infty}(\R^6)} |\int_{\R^6} \phi(z, x) \{d\mu_s^{(n)} - d\mu_s^{(n-1)}\}|ds.
\end{align*}
We can take supremum on the left side over $h \in C_b^{\infty}(\R^6)$ and call 
the resulting expression as $J_{n+1}(t)$.
Then we have 
\[
J_{n+1}(t) \le K\e + C_1\int_0^t J_{n+1}(s)ds + \int_0^t J_n(s)ds.
\]
By the Gronwall inequality, 
\begin{align*}
J_{n+1}(t) & \le K\e e^{C_1t} + C_1e^{C_1t}\int_0^t J_n(s)e^{-C_1s}ds\\
& \le K\e e^{C_1t} + C_1e^{C_1t}\int_0^te^{-C_1s}\left[ K\e e^{C_1s} + C_1e^{C_1s}\int_0^s J_{n-1}(r) e^{-C_1 r}dr\right]\\
&\le \cdots \\
&\le K\e e^{C_1 t} (1 + C_1t + \frac{C_1^2t^2}{2!} + \cdots 
 + \frac{(C_1 t)^{n-1}}{(n-1)!})\\
& + C_1^n e^{C_1 t} \int_0^t \int_0^{s_1}\cdots \int_0^{s_{n-1}}J_1(r)drds_{n-1}\cdots ds_1
\end{align*}
which tends to zero as $n \to \infty$ and $\e \to 0$. 
 \end{proof}
 
  \section{Identification of the Limit}
   
   In this section, we will conclude the proof of Theorem \ref{thm-existence-vel} by using tightness,  Proposition \ref{prop5.1} and 
   convergence of martingale problems.
    From the tightness of $\{\mu^{(n)}\}$, we have the existence of  a weakly convergent subsequence  $\{\mu^{(n_k)}\}$. 
     Let its weak limit be denoted by $\mu$. Consider the associated subsequence $\{\mu^{(n_k +1)}\}$ which is also tight so that there exists a further subsequence $\{\mu^{(n_{k_j} +1)}\}$ which converges weakly. We will call its limit as $\nu$. Clearly,  
     $\{\mu^{n_{k_j}}\}$ being a subsequence of $\{\mu^{(n_k)}\}$ converges weakly to $\mu$. Our aim in this section is to identify 
      $\mu$ as a weak solution of the Enskog equation.
     
     \sni
     Let us denote a generic element of $\mathbb{D}\times \mathbb{D}$ as $\omega_1 \times \omega_2$. In the canonical setup 
     on the path space $\mathbb{D}\times \mathbb{D}$, we recall that $\mu^{(n_{k_j} +1)}$ is the solution of the following martingale problem:
      
      \noindent
      For any  function $\phi \in C_b^{2}(\R^3 \times \R^3)$, and any $t \in [0, T]$, 
      \begin{itemize}
      \item[(i)] $\mu_0^{(n_{k_j} +1)} = \mathcal{L}(Z_0, X_0)$ (specified).
      \begin{align*}
    &\hspace{-25pt}(ii)\;\;\; \phi(\omega_1(t), \omega_2(t)) - \phi(\omega_1(0), \omega_2(0)) - \int_0^t \nabla_x\phi(\omega_1(s), \omega_2(s)) \cdot \omega_1(s) 
      ds \\ 
      & - \int_0^t \int_{U_0} \{\phi( \omega_1(s) + \a(\omega_1(s), v(s), \xi), \omega_2(s)) - \phi(\omega_1(s), \omega_2(s))\} \\
      &\;\;\;\;\;\;\;\;\s(|\omega_1(s) 
       - v_s|) \beta(|\omega_2(s) - y_s|) Q(d\xi)\mu^{(n_{k_j})}(dv, dy) ds
       \end{align*}
       is a $\mu^{(n_{k_j} +1)}$-martingale.
      \end{itemize}
     
     \sni
     While it is important to keep the above setup of martingale problems in mind, we will pass on to construct convenient random 
     processes (on a possibly different probability space)  for ease in 
     calculations.  Given that $\mu^{(n_{k_j} + 1)} \to \nu$ and $\mu^{(n_{k_j})} \to \mu$, by the Skorohod representation theorem, we can construct random processes $(Z^{(n_{k_j} + 1)}, X^{(n_{k_j} + 1)})$ 
     and $(Z, X)$ such that $\mathcal{L} (Z^{(n_{k_j} + 1)}, X^{(n_{k_j} + 1)}) = \mu^{(n_{k_j}+1)}$ and 
     $\mathcal{L}(Z, X) = \nu$, and 
     \[(Z^{(n_{k_j} + 1)}, X^{(n_{k_j} + 1)}) \to (Z, X) \;\;\text{a.s.}
     \]
     Independently of $(Z^{(n_{k_j} + 1)}, X^{(n_{k_j} + 1)})$, we  construct processes 
     $(\tilde{Z}^{(n_{k_j})}, \tilde{X}^{(n_{k_j})})$ and $(\tilde{Z}, \tilde{X})$ with 
     \[
     \mathcal{L}(\tilde{Z}^{(n_{k_j})}, \tilde{X}^{(n_{k_j})}) = \mu^{(n_{k_j})};\;\;\mathcal{L}(\tilde{Z}, \tilde{X}) = \mu
     \]
     such that $(\tilde{Z}^{(n_{k_j})}, \tilde{X}^{(n_{k_j})}) \to (\tilde{Z}, \tilde{X})$ a.s. The latter processes are used in this 
     section only to shorten certain expressions and write them in more convenient forms to aid calculations. In terms 
     of $(Z^{(n_{k_j} + 1)}, X^{(n_{k_j} + 1)})$, we are able to write the requirement (ii) in the statement of the martingale problem 
     as follows:
     \noindent
     Fix any finite integer $r$. For any $0 \le s_1 \le \cdots \le s_r \le s < t \le T$, and any choice of bounded $\mathcal{F}_{s_i}$ functions $g_i$
     for all $i = 1, 2, \cdots , r$, we require 
  \begin{align}
  &\mathbb{E} \left[ (\phi(Z_t^{(n_{k_j} + 1)}, X_t^{(n_{k_j} + 1)}) - \phi(Z_s^{(n_{k_j} + 1)}, X_s^{(n_{k_j} + 1)}) - \int_s^t \nabla_x\phi(Z_u^{(n_{k_j} + 1)}, X_u^{(n_{k_j} + 1)}) \cdot  Z_u^{(n_{k_j} + 1)}du \right.\notag\\ 
  & - \int_s^t \int_{U_0} \{\phi(Z_u^{(n_{k_j} + 1)}) + \a(Z_u^{(n_{k_j} + 1)}, v_u, \xi), X_u^{(n_{k_j} + 1)}) - \phi(Z_u^{(n_{k_j} + 1)}, X_u^{(n_{k_j} + 1)})\} \notag\\
  &\left. \;\;\;\;\;\;\;\;\s(|Z_u^{(n_{k_j} + 1)} - v_u|) \beta(|X_u^{(n_{k_j} + 1)} - y_u|) Q(d\xi)\mu^{(n_{k_j})}(dv, dy) du)\Pi_{i=1}^r g_i\right]
   = 0 \label{mgpr}
  \end{align}
   By letting $j \to \infty$, we know that  
   \begin{equation}
   \mathbb{E} |\phi(Z_t^{(n_{k_j} + 1)}, X_t^{(n_{k_j} + 1)})]  - \phi(Z_t,  X_t)|\to 0 \label{term1}
   \end{equation}
   if $t \in D$ where $D:= \{a \in [0, T]: \nu(\Delta (Z_a,  X_a) \not= (0, 0)) = 0\}$.  A similar 
   statement holds when $t$ is replaced by $s$ provided that $s \in D$. 
   The statement that 
   \begin{equation}
   \mathbb{E}|\int_s^t \nabla_x\phi(Z_u^{(n_{k_j} + 1)}, X_u^{(n_{k_j} + 1)}) \cdot  Z_u^{(n_{k_j} + 1)}du  -  \int_s^t \nabla_x\phi(Z_u, X_u) \cdot  Z_udu| 
   \to 0\label{term2}
   \end{equation}
   follows from the $L^2(P)$ boundedness of $\sup_{0 \le u \le T} |Z_u^{(n_{k_j} + 1)}|$ indexed by $j$. Hence, our next objective is to show that the last term 
   on the left side of \eqref{mgpr} converges to a limit.

  For $\phi \in C_b^{2}(\R^3 \times \R^3)$, we will use the notation
    \[
    \norm{\phi_z^{\pr}}_{\infty} = \left(\sum_{i=1}^3 |\frac{\partial \phi}{\p z_i}|^2\right)^{1/2}; \;\;\;\;\norm{\phi_{zz}^{\dpr}}_{\infty} = 
    \left(\sum_{i, j=1}^3 |\frac{\partial^2 \phi}{\p z_i \p z_j}|^2\right)^{1/2}
    \]
    with similar meanings for $\norm{\phi_x^{\pr}}_{\infty}, \norm{\phi_{xx}^{\dpr}}_{\infty}$, and $\norm{\phi_{zx}^{\dpr}}_{\infty}$.

     \sni
     Fix any function $\phi \in  C_b^2(\R^3)$. On $\R^{12} \times \Xi$, define
     \[
     G(z, x, v, y, \xi) =  \{\phi(z + \a(z, v, \xi), x) - \phi(z, x)\} \s(|z - v|)\beta(|x - y|).
     \]
     \noindent
     {\sc Claim:} 
     \begin{align*}
     &\int_0^t |\int_{U_0} G(Z_{s}^{(n_{k_j} + 1)},  X_{s }^{(n_{k_j} + 1)},  v_s, y_s, \xi) \mu^{(n_{k_j})}(dv, dy) dQ\\
    & \;\;\;\; - \int_{U_0} G(Z_s, X_s, v_s, y_s, \xi) \mu(dv, dy) dQ|ds 
    \end{align*}
    converges to  $0$ almost surely.
    \begin{proof}
    Denoting the above expression by 
    $D(Z_s^{(n_{k_j} + 1)}, X_s^{(n_{k_j} + 1)}, Z_s, X_s, \mu^{(n_{k_j})}, \mu)$, we have
    \begin{align*}
     &D(Z_s^{(n_{k_j} + 1)}, X_s^{(n_{k_j} + 1)}, Z_s, X_s, \mu^{(n_{k_j})}, \mu)\\
    & \le D(Z_s^{(n_{k_j} + 1)}, X_s^{(n_{k_j} + 1)}, Z_s, X_s, \mu^{(n_{k_j})}, \mu^{(n_{k_j})}) + D(Z_s, X_s, Z_s, X_s, \mu^{(n_{k_j})}, \mu)\\
    &= I_1 + I_2
    \end{align*}
    for short.  Then, $I_1 \le J_1 + J_2$ where  
    \begin{align*}
    J_1 
    &= \int_0^t \vert \int_{U_0} [ \{ \phi(Z_s^{(n_{k_j} + 1)} + \a(Z_s^{(n_{k_j} + 1)}, v_s, \xi), X_s^{(n_{k_j} + 1)}) - \phi(Z_s^{(n_{k_j} + 1)}, X_s^{(n_{k_j} + 1)})\}\\
    & \;\;\;\; - \{\phi(Z_s + \a(Z_s, v_s, \xi), X_s) - \phi(Z_s, X_s)\}] \\
    & \;\;\;\;\s(|Z_s^{(n_{k_j} + 1)} - v_s|) \beta(|X_s^{(n_{k_j} + 1)} - y_s|)dQ \mu^{(n_{k_j})}(dv, dy)\vert ds, \;\;\text{and} 
    \end{align*}
    \begin{align*}
    J_2 &= \int_0^t \vert \int_{U_0}  \{ \phi(Z_s + \a(Z_s, v_s, \xi), X_s) - \phi(Z_s, X_s)\} \\
    & \;\;\;\;\;\;\{\s(|Z_s^{(n_{k_j} + 1)} - v_s|) \beta(| X_s^{(n_{k_j} + 1)} - y_s|)  - \s(|Z_s - v_s|) \beta(| X_s - y_s|) \}dQ d\mu^{(n_{k_j})}\vert ds
    \end{align*}
    In order to bound $J_1$, we bound $\s$ and $\beta$ by $1$ in the above expression, and then,  we break up the rest of the integrand appearing in $J_1$ as follows, and estimate 
    each part separately.   
    \begin{align*}
    J_1 &\le \int_0^t \int_{\mathbb{D}^2}  \int_{\Xi} \vert [\{\phi(Z_s^{(n_{k_j} + 1)} + \a(Z_s^{(n_{k_j} + 1)}, v_s, \xi), X_s^{(n_{k_j} + 1)}) - \phi(Z_s^{(n_{k_j} + 1)}, X_s^{(n_{k_j} + 1)})\}\\
    & \;\;\;\; - \{\phi(Z_s + \a(Z_s^{(n_{k_j} + 1)}, v_s, \xi), X_s^{(n_{k_j} + 1)}) - \phi(Z_s, X_s^{(n_{k_j} + 1)})\}]\\
    & \;\;\;\; + [\{\phi(Z_s + \a(Z_s^{(n_{k_j} + 1)}, v_s, \xi), X_s^{(n_{k_j} + 1)}) - \phi(Z_s, X_s^{(n_{k_j} + 1)})\} \\
    & \;\;\;\; - \{\phi(Z_s + \a(Z_s, v_s, \xi), X_s^{(n_{k_j} + 1)}) - \phi(Z_s, X_s^{(n_{k_j} + 1)})\}]\\
    & \;\;\;\; + [\{\phi(Z_s + \a(Z_s, v_s, \xi), X_s^{(n_{k_j} + 1)}) - \phi(Z_s, X_s^{(n_{k_j} + 1)})\}\\
    & \;\;\;\; - \{\phi(Z_s + \a(Z_s, v_s, \xi), X_s) - \phi(Z_s, X_s)\}]\vert  dQ  d\mu^{(n_{k_j})}ds\\
    & \le \int_0^t \int_{U_0} [ \norm{\phi_{zz}^{\dpr}}_{\infty} \vert \a(Z_s^{(n_{k_j} + 1)}, v_s, \xi)| |Z_s^{(n_{k_j} + 1)} - Z_s \vert \\
    &\;\;\;\;\;\;\;\;  + \norm{\phi_z^{\pr}}_{\infty}
    \vert \a(Z_s^{(n_{k_j} + 1)}, v_s, \xi) - a(Z_s, v_s, \xi)\vert \\
    &\;\;\;\;\;\;\;\; + \norm{\phi_{zx}^{\dpr}}_{\infty} |\a(Z_s, v_s, \xi)| |X_s^{n_{k_j} + 1} - X_s|]dQ \mu^{(n_{k_j})}(dv, dy)ds
    \end{align*} 
    which tends to zero as $j \to \infty$, by restricting all processes to $B_R$ as in the previous section, and using the bounded convergence theorem. Now, we consider 
    $J_2$.
    \begin{align*}
    J_2 
    &= \int_0^t \int_{U_0} \vert \{ \phi(Z_s + \a(Z_s, v_s, \xi), X_s) - \phi(Z_s, X_s)\} \\
    & \{\s(|Z_s^{(n_{k_j} + 1)} - v_s|) \beta(| X_s^{(n_{k_j} + 1)} - y_s|) \\
    &\;\;\;\;\;\;\;\; - \s(|Z_s - v_s|) \beta(| X_s - y_s|) \}\vert dQ d\mu^{(n_{k_j})}|ds\\
    & \le \int_0^t \int_{U_0}  \norm{\phi_z^{\pr}}_{\infty} |\a(Z_s, v_s, \xi)| \{|Z_s^{(n_{k_j} + 1)} - Z_s| + | X_s^{(n_{k_j} + 1)} - X_s|\} dQ d\mu^{(n_{k_j})}ds\\
    &\le C_1 \int_0^t \int_{\mathbb{D}^2} (|Z_s| + |v_s|) \{|Z_s^{(n_{k_j} + 1)} - Z_s| + | X_s^{(n_{k_j} + 1)} - X_s|\}d\mu^{(n_{k_j})}ds
    \end{align*}
    with $C_1$ as a suitable constant, and the last expression above  $\to 0$ as $j \to \infty$ using arguments as before. 
    
    \sni
    Next, we consider $I_2$ where
    \begin{align*}
    I_2  &= \int_0^t \vert \int_{U_0} \{\phi(Z_s + \a(Z_s, v_s, \xi), X_s) - \phi(Z_s, X_s)\}\\
    &  \s(|Z_s - v_s|) \beta(|X_s - y_s|) dQ (d\mu^{(n_{k_j})} - d\mu) \vert ds.
    \end{align*}
    We write the expression within absolute value as 
    \begin{align*}
    \tilde{\mathbb{E}}& \int_{\Xi} [ \{ \phi(Z_s + \a(Z_s, \tilde{Z}_s^{(n_{k_j})}, \xi), X_s) - \phi(Z_s, X_s)\} \s(|Z_s - \tilde{Z}_s^{(n_{k_j})}|) \beta(|X_s - \tilde{X}_s^{(n_{k_j})}|)\\
     & - \{ \phi(Z_s + \a(Z_s,  \tilde{Z}_s, \xi), X_s) - \phi(Z_s, X_s)\} \s(|Z_s - \tilde{Z}_s|) \beta(|X_s - \tilde{X}_s|)]dQ.
    \end{align*}
    where $\tilde{\mathbb{E}}$ refers to expectation with respect to the random variables  $\tilde{Z}_s^{(n_{k_j})} \tilde{X}_s^{(n_{k_j})}, \tilde{Z}_s$, and $\tilde{X}_s$.  Using this, 
    \begin{align*}
    I_2 & \le \int_0^t  \tilde{\mathbb{E}} [ C \norm{\phi_z^{\pr}}_{\infty}\{ | \tilde{Z}_s^{(n_{k_j})} - \tilde{Z}_s| +  \norm{\phi}_{\infty}(|Z_s| + |\tilde{Z}_s| )\\
    &\;\;\;\;\;\;\;\;\times (|\tilde{Z}_s^{(n_{k_j})} - \tilde{Z}_s| + | \tilde{X}_s ^{(n_{k_j})} - \tilde{X}_s|)\}]ds
    \end{align*} 
    which goes to zero as $j \to \infty$ boundedly by restricting the processes appearing in the above expression to a compact set. By bounded convergence theorem, $I_2 \to 0$.
    This finishes the proof of  Claim.
    \end{proof}
    Next, we observe that 
    \begin{align*}
    &\mathbb{E}  \sup_{0 \le t \le T}  | \int_0^t \int_{U_0} \{\phi(Z_s^{(n_{k_j} + 1)} + \a(Z_s^{(n_{k_j} + 1)}, v_s, \xi), X_s^{(n_{k_j} + 1)}) - \phi(Z_s^{(n_{k_j} + 1)}, X_s^{(n_{k_j} + 1)})\} \\
    & \;\;\;\;\;\;\;\;\s(|Z_s^{n_{k_j} + 1} - v_s|)\beta(|X_s^{(n_{k_j} + 1)} -  y_s|) dQ d\mu^{(n_{k_j})}|^2 ds\\
    & \le \norm{\phi_z^{\pr}}_{\infty}^2 \int_0^T \int_{\mathbb{D}^2} \mathbb{E} (|Z_s^{(n_{k_j} +1)}|^2  + |v_s|^2)d\mu^{(n_{k_j})}ds\\
    & \le C
    \end{align*}
    for a suitable constant $C$ since 
    $
    \mathbb{E} \left(\displaystyle{\sup_{0 \le t \le T}}[ |Z_s^{(n_{k_j} +1)}|^2  + |Z_s^{(n_{k_j})}|^2 ]\right) \le C_1
    $
    
    for a constant $C_1$ which is independent of $n_{k_j}$. 
    \begin{align}
    & \mathbb{E} [\int_0^t |\int_{U_0} G(Z_{s}^{(n_{k_j} + 1)},  X_{s }^{(n_{k_j} + 1)},  v_s, y_s, \xi) \mu^{(n_{k_j})}(dv, dy) dQ\notag\\
    &\;\;\;\; - \int_{U_0} G(Z_s, X_s, v_s, y_s, \xi) \mu(dv, dy) dQ|ds] \;\to 0 \label{term3}
    \end{align}
    By \eqref{term1} - \eqref{term3}, we conclude that the martingale problem posed by the $(n_{k_j} + 1)$st  stochastic differential equation converges to a solution of the martingale problem posed by the stochastic system
  \begin {align}
   Z_t^{\mu} &= Z_0\notag\\
   &+ \int_0^t \int_{U}\alpha(Z_{s}^{\mu}, v_s, \xi)1_{[0,\; \s(|Z_{s}^{\mu} - v_s|)\beta(|X_s^{\mu} - y_s|)]}(r) d\tilde {N}^{\mu}\notag\\
   & + \int_{U_0} \hat{\alpha}(Z_s^{\mu}, v_s, \xi)\beta(|X_s^{\mu} - y_s|) \mu(dy dv) Q(d\xi)ds \label {mu-eq} \\
   X_t^{\mu}& = X_0+\int_0^t Z_s^{\mu}ds,  \label{mu-eq-spacer}.
   \end{align}
    We know from Section 3 that the above system, with $\mu$ given,  has a unique solution.  Since $\mu^{(n_{k_j} + 1)}$ converges in law to $\nu$, and the law of $(Z^{\mu}, X^{\mu})$ is given by $\nu$.  By the result in Section 5, we know that if $\mu^{(n_{k_j})}$ converges in law, then 
     $\mu^{(n_{k_j} +1)}$   also converges in law to the same limit. Hence $\nu  = \mu$, and $\mu$ is a solution of the martingale problem posed by \eqref{mu-eq} and \eqref{mu-eq-spacer}. The proof of Theorem \ref{thm-existence-vel} is thus completed.

\section{Uniqueness of Solutions}
In this section, we will study the uniqueness of solutions to the martingale problem posed by the Enskog equation. Uniqueness is taken in the sense of uniqueness of time marginals of the distribution of the processes $(X, Z)$. 
We consider the Enskog equation under the additional hypothesis that the law of any weak solution of the limit equation 
admits at each time point $t$, a density with respect to the Lebesgue measure on $\R^6$. Such a hypothesis can likely be replaced by 
imposing certain conditions on the functions $\a$ and $\s$ that are expedient though it would take us far away from the physics of the problem.  

\begin{thm}\label{ThmUniq}
Let the  hypotheses used in Theorem \ref{thm-existence-vel} hold. In addition, let $\s$ be in $C_b^{\infty}(\R^1)$.   
Then, for any fixed $t$ in the interval $[0, T]$, the $t$-marginal distribution of weak solutions of the limit equation (given below) is unique within the class of 
Borel probability measures on $\mathbb{R}^6$ that are absolutely continuous with respect to the Lebesgue measure on $\R^6$. 

 \begin{align}
 Z_t & = Z_0 + \int_0^t \int_{\mathbb{D}^2\times \Xi} \alpha(Z_s, v_s, \xi) 1_{_{[0, \s(|Z_s - v_s|)\beta(|X_s - y_s|)]}}(r)
\tilde {N}(dv, dy, d\xi, ds)\nonumber\\
&+ \int_0^t \int_{U_0} \alpha(Z_s, v_s, \xi)\s(|Z_s - v_s|)\beta(|X_s - y_s|) \mu(dy, dv) Q(d\xi) ds
\label{leqn1}\\
X_t &= X_0 + \int_0^t Z_sds\label{leqn2}
\end{align} 
wherein the law of $(X, Z)$ is given by $\mu$,  and the compensator of $\tilde {N}$ is given by 
$\mu(dy, dv)Q(d\xi)dt$. 
\end{thm}

\begin{proof}  
Let  $h$ be in the space $L:= C_b^{\infty}(\R^3 \times \R^3)$. 
By the It\^o formula, we have
\begin{align}
 \mathbb{E} [h(X_t, Z_t)]  &= \mathbb{E} [h(X_0, Z_0)]  + \mathbb{E}\int_0^t \nabla_x h(X_s, Z_s) \cdot Z_sds\nonumber\\
 &+ \mathbb{E}[\int_0^t \int_{{\mathbb{D}^2}\times \Xi} \{h(X_s, Z_s + \a (Z_s, v_s, \xi)) - h(X_s, Z_s)\}\nonumber\\
 &\;\;\;\;\;\;\;\;\;\;\;\;\;\;\;\;\;\;\;\;\;\;\;\;\;\;\s(|Z_s - v_s|)
 \beta(|X_s - y_s|)\mu(dy, dv)Q(d\xi)ds]. \label{eq3.1}
\end{align}
If there exists another solution of \eqref{leqn1} and \eqref{leqn2}, let us denote it as $(X\pr, Z\pr)$ with its law denoted by $\rho$.  
For each $s \in [0, T]$, the measures $\mu_s$ and $\rho_s$ denote the $s$-marginals of $\mu$ and $\rho$ respectively. 

\par
Using the fact that both $(X_t, Z_t), (X\pr_t, Z\pr_t)$ are  solutions of \eqref{leqn1} and \eqref{leqn2}, and using $\mathbb{E}$ to denote 
expectation of functions of $(X_t, Z_t)$ as well as $ (X\pr_t, Z\pr_t)$, we have
\begin{align*}
& \mathbb{E} h(X_t, Z_t) - \mathbb{E} h(X\pr_t, Z\pr_t)\\
&= \mathbb{E}\int_0^t\left[\nabla_x h(X_s, Z_s) \cdot Z_s - \nabla_x h(X\pr_s, Z\pr_s) \cdot Z\pr_s\right]ds\\
&+ \mathbb{E}[\int_0^t \int_{\mathbb{D}^2}\int_{\Xi}  \{h(X_s, Z_s + \a (Z_s, v_s, \xi)) - h(X_s, Z_s)\}\s(|Z_s - v_s|)\\
&\;\;\;\;\;\;\;\;\;\;\;\;\;\;\;\;\;\;\;\;\;\;\;\;\beta(|X_s - y_s|)\mu(dy, dv)
Q(d\xi)ds] \\
&\;\;\;- \mathbb{E}[\int_0^t \int_{\mathbb{D}^2}\int_{\Xi}  \{h(X\pr_s, Z\pr_s + \a (Z\pr_s, v\pr_s, \xi)) - h(X\pr_s, Z\pr_s)\}\s(|Z\pr_s - v\pr_s|)\\
&\;\;\;\;\;\;\;\;\;\;\;\;\;\;\;\;\;\;\;\;\;\;\;\;\beta(|X\pr_s - y\pr_s|)\rho(dy\pr, dv\pr)Q(d\xi)ds]\\
\end{align*}
\begin{align*}
&= \mathbb{E}\int_0^t\left[\nabla_x h(X_s, Z_s) \cdot Z_s - \nabla_x h(X\pr_s, Z\pr_s) \cdot Z\pr_s\right]ds\\
& + \int_0^t \int_{\mathbb{D}^2}\int_{\mathbb{D}^2}\int_{\Xi}  \{h(x_s, z_s + \a (z_s, v_s, \xi)) - h(x_s, z_s)\}\\
&\;\;\;\;\;\;\;\;\;\;\;\;\;\;\;\;\;\;\;\;\;\;\;\;\s(|z_s - v_s|)
\beta(|x_s - y_s|)\mu(dx, dz)\mu(dy, dv)Q(d\xi)ds\\
&\;\;\; -  \int_0^t \int_{\mathbb{D}^2}\int_{\mathbb{D}^2}\int_{\Xi}  \{h(x\pr_s, z\pr_s + \a (z\pr_s, \Pi_s(v\pr), \xi)) - h(x\pr_s, z\pr_s)\}\\
&\;\;\;\;\;\;\;\;\;\;\;\;\;\;\;\;\;\;\;\;\;\;\;\;\s(|z\pr_s - v\pr_s|)\beta(|x\pr_s - y\pr_s|)\rho(dx\pr, dz\pr)\rho(dy\pr, dv\pr)Q(d\xi)ds\\
& = D_0 + D_1.
\end{align*}
where 
\begin{equation}
D_0 := \mathbb{E}\int_0^t\left[\nabla_x h(X_s, Z_s) \cdot Z_s - \nabla_x h(X\pr_s, Z\pr_s) \cdot Z\pr_s\right]ds\label{d0}
\end{equation}
\begin{align}
D_1 &:= \int_0^t \int_{\mathbb{D}^2}\int_{\mathbb{D}^2}\int_{\Xi}  [h(z_s + \a (z_s, v_s, \xi)) h(z_s)]\s(|z_s - v_s|)\nonumber\\
&\;\;\;\;\;\;\;\;\;\;\;\;\; \beta(|x_s - y_s|)
\mu(dy, dv)\mu(dx, dz)Q(d\xi)ds\nonumber\\
& - \int_0^t \int_{\mathbb{D}^2}\int_{\mathbb{D}^2}\int_{\Xi}  [h(z\pr_s + \a (z\pr_s, v\pr_s, \xi)) - h(z_s)] \s(|z\pr_s - v\pr_s|)\nonumber\\
&\;\;\;\;\;\;\;\;\;\;\;\;\;
\beta(|x\pr_s - y\pr_s|)\rho(dy\pr, dv\pr)\rho(dx\pr, dz\pr)Q(d\xi)ds.\label{d1}
\end{align}

We will first obtain an upper bound for $D_0$.  Define $f(x, z) := \nabla_x h(x, z)\cdot z$. 

Fix any $R > 0$. If $|z| <R$, then $f$ is in the space $L$. By existence of the second moments of $\sup_{0 \le t \le T}|Z_t|$ and 
$\sup_{0 \le t \le T}|Z_t\pr|$, we obtain that
for any fixed $t \in [0, T]$, given any $\epsilon > 0$, we can find a compact set $K$ in $\R^6$ such that 
\[
\int_{K^c} |f(x, z)| \{\mu_t(dx, dz) + \rho_t(dx, dz)\} < \frac{\epsilon}{2T}
\]
Therefore, we have 
\[
|D_0| \le \e/2 +\int_0^t |\int_K f(x, z) \{\mu_t(dx, dz) - \rho_t(dx, dz)\}|dt
\]
Let $R$ be large enough so that $K \subset B_R$ where $B_R$ denotes the $R$-ball in $\R^6$. The function $f(x, z)1_K(x, z)$ can be 
approximated by a function $f_{\delta}$ in $L$ by replacing $1_K$ by a smooth non-negative function which is identically equal to $1$ on $K$ and whose support 
is in the $\delta$ neighborhood of $K$, $\d > 0$,  which is contained in $B_R$. As $\delta$ tends to $0$, $f_{\delta} \to f1_K$, pointwise. By the bounded convergence theorem, 
\[
\lim_{\delta \to 0} \int_0^t \int_K |f(x, z) - f_{\delta}(x, z)| \{\mu_t(dx, dz) + \rho_t(dx, dz)\}dt = 0.
\]
Hence, for all $\delta$ small enough, 
\[
\int_0^t \int_K |f(x, z) - f_{\delta}(x, z)| \{\mu_t(dx, dz) + \rho_t(dx, dz)\}dt < \e/2.
\]
The above estimates allow us to conclude that 
\[
D_0 \le \e + \int_0^t |\int f_{\delta}(x, z)\{\mu_t(dx, dz) - \rho_t(dx, dz)\}|dt.
\]
Since $ f_{\delta} \in L$, the above inequality ipso facto yields  
\[
|D_0| \le \e + \int_0^t \sup_{\phi \in L} |\int \phi(x, z) \{\mu_t(dx, dz) - \rho_t(dx, dz)\}|dt.
\]
Next, we proceed to estimate $D_1$. 
On the right side of equation \eqref{d1}, we add and subtract 
\begin{align*}
&\int_0^t \int_{\mathbb{D}^2}\int_{\mathbb{D}^2}\int_{\Xi}  [h(x_s, z_s + \a(z_s, v_s, \xi)) - h(x_s, z_s)] \s(|z_s - v_s|)\beta(|x_s - y_s|)\nonumber\\
&\;\;\;\;\;\;\;\;\;\;\;\;\;\mu(dy, dv)\rho(dx, dz)Q(d\xi)ds
\end{align*}
This enables us to split $D_1$ and write it as $G_1 + G_2$, 
 where 
\begin{align*}
G_1 &:= \int_0^t \int_{\mathbb{D}^2}\int_{\mathbb{D}^2}\int_{\Xi}  [h(x_s, z_s + \a(z_s, v_s, \xi)) - h(x_s, z_s)]\s(|z_s - v_s|)\beta(|x_s - y_s|)\nonumber\\
&\;\;\;\;\;\;\;\;\;\;\;\;\; \mu(dx, dz)\mu(dy, dv)Q(d\xi)ds\nonumber\\
& - \int_0^t \int_{\mathbb{D}^2}\int_{\mathbb{D}^2}\int_{\Xi}  [h(x_s, z_s + \a(z_s, v_s, \xi)) - h(x_s, z_s)] \s(|z_s - v_s|)\beta(|x_s - y_s|)\nonumber\\
&\;\;\;\;\;\;\;\;\;\;\;\;\;\mu(dy, dv)\rho(dx, dz)Q(d\xi)ds.
\end{align*}
and 
\begin{align*}
G_2 &:= \int_0^t \int_{\mathbb{D}^2}\int_{\mathbb{D}^2}\int_{\Xi}  [h(x_s, z_s + \a(z_s, v\pr_s, \xi))  - h(x_s, z_s)] \s(|z_s - v\pr_s|)\beta(|x_s - y\pr_s|)\nonumber\\
&\;\;\;\;\;\;\;\;\;\;\;\;\;\mu(dy\pr, dv\pr)\rho(dx, dz)Q(d\xi)ds\nonumber\\
& -  \int_0^t \int_{\mathbb{D}^2}\int_{\mathbb{D}^2}\int_{\Xi} [h(x\pr_s, z\pr_s + \a(z\pr_s, v\pr_s, \xi))  -  h(x\pr_s, z\pr_s)] \s(|z\pr_s - v\pr_s|)\beta(|x\pr_s - y\pr_s|)\nonumber\\
&\;\;\;\;\;\;\;\;\;\;\;\;\;\rho(dy\pr, dv\pr)\rho(dx\pr, dz\pr)Q(d\xi)ds.
\end{align*}
To find an upper bound for $G_1$, first introduce, for each $s \in [0, T]$, the function $f_s$ defined on $\R^6$ by
\begin{equation*}
f_s(x, z) = \int_{\R^6\times \Xi} [h(x, z + \a(z, v, \xi)) - h(x, z)]\s(|z-v|)\beta(|x-y|) Q(d\xi)\mu_s(dy, dv).
\end{equation*}
Then $G_1$ can be written as 
\[
\int_0^t \int_{\mathbb{D}^2}f_s(x_s, z_s) (\mu(dx, dz) - \rho(dx, dz)).
\]
We will now use the assumption in the statement of the theorem
 that for each $s$, $\mu_s << \lambda$, and $\rho_s << \lambda$,  where $\lambda$ is the Lebesgue 
measure on $\R^6$. Hence the product measures $\mu_s\times \mu_s$ and $\mu_s\times \rho_s$ are both 
absolutely continuous with respect to the Lebesgue measure on $\R^{12}$. Therefore, given $\e > 0$, there exists a 
$\d$-neighborhood $G:= G_{\d}$ of the diagonal in $\R^6\times \R^6$ such that 
\[
(\mu_s\times\mu_s + \mu_s\times\rho_s)(G) < \e
\]
For a suitable constant $C > 0$, we have
\begin{align}
| G_1| &\le  C\e + \int_{\R^6} \int_{\R^6\times \Xi} 1_{G^c}(x, z, y, v)[h(x, z + \a(z, v, \xi)) - h(x, z)]\s(|z-v|)\nonumber\\
&\;\;\;\;\;\;\;\;\;\;\;\;\;\;\;\;\;\;\;\;\;\;\;\;\beta(|x-y|) Q(d\xi)\mu_s(dy, dv)(\mu_s(dx, dz) - \rho_s(dx, dz)). \label{g1.1}
\end{align}
We  approximate the function $1_{G^c}$ by a smooth bounded function $\phi$ which takes the value $1$
in $G^c$, and  zero in $G_{\d/2}$. With such a choice of $\phi$, we have
\begin{align}
\int_{\R^{12}\times \Xi} &|\{1_{G^c}(y, v) - \phi\}(x, z, y, v)[h(x, z + \a(z, v, \xi)) - h(x, z)]\nonumber\\
&\;\;\;\;s(|z-v|)\beta(|x-y|) Q(d\xi)|\mu_s(dy, dv)\{\mu_s(dx, dz) + \rho(dx, dz)\}\nonumber\\
 & <  C\e.\label{g1.2}
 \end{align}
 By \eqref{g1.1} and \eqref{g1.2}, it follows that
 \begin{align}
 |G_1| &\le 2C\e + |\int_{\R^6} \int_{\R^6\times \Xi} \phi(x, z, y, v)[h(x, z + \a(z, v, \xi)) - h(x, z)]\\
 &\;\;\;\;\;\;\;\;\;\;\;\; \s(|z-v|)\beta(|x-y|) Q(d\xi)|\mu_s(dy, dv)(\mu_s(dx, dz) - \rho_s(dx, dz))|.\label{g1.3}
 \end{align}
 The integral
 \[
 \int_{\R^6\times \Xi} \phi(x, z, y, v)[h(x, z + \a(z, v, \xi)) - h(x, z)]\s(|z-v|)\beta(|x-y|) Q(d\xi)|\mu_s(dy, dv)
 \] is smooth and bounded 
  as a function of $x, z$, and therefore it is in $L$. Hence, \eqref{g1.3} yields
  \begin{equation}
  |G_1| \le 2C\e +  \int_0^t \sup_{\psi \in L} |\int \psi(x, z) \{\mu_t(dx, dz) - \rho_t(dx, dz)\}|dt.\label{g1.4}
  \end{equation}
 
To bound $G_2$, we introduce the function $g_s$ on $\R^6$ by 
\[
g_s(y, v) := \int [h(x, z + \alpha(z, v, \xi)) - h(x, z)] \s(|z - v|) \beta(|x - y|)Q(d\xi) \rho_s(dx, dz).
\]
Following the ideas used in finding an upper bound for $G_1$, one can, mutatis mutandis,  bound  $G_2$, and conclude that 
\begin{equation}
|G_2| \le  2C\e  + \int_0^t \sup_{\psi \in L} |\int \psi(y, v) \{\mu_t(dy, dv) - \rho_t(dy, dv)\}|dt.\label{g2.1}
\end{equation}

\par
In summary, we have for any $\e > 0$, 
\begin{equation}
|E[h(X_t, Z_t) - h(X\pr_t, Z\pr_t)]| \le  5C\e +  3\int_0^t \sup_{\phi \in L}|\mathbb{E}[\phi(X_s, Z_s)] - \mathbb{E}[\phi(X\pr_s, Z\pr_s)]|\,ds
\end{equation}
for all $ h \in L$, for a suitable constant $C > 1$.  With $\e$ being arbitrary, let $\e \to 0$. By the Gronwall  Lemma, the uniqueness of the time-marginal distributions of weak solutions is obtained. 
\end{proof} 
\begin{Remark}
In the context of martingale problems, uniqueness of distribution of the time marginals of solutions implies  uniqueness of the law on the path space (\cite{KSu}, page 69). Hence, Theorem \ref{ThmUniq} gives us uniqueness in law of the solution of \eqref{leqn1} and 
\eqref{leqn2} on $\mathbb{D}\times \mathbb{D}$.
\end{Remark}

\section{Invariant Gaussian density for velocity}
Let $\{X_s, Z_s\}_{s\in \mathbb{R^+}}$ be a solution of   \eqref{eq-vel},\eqref{eq-space}, which corresponds to the Enskog equation in the kinetic theory of gases. Let  MVN $(0, I)$ denote the  standard normal distribution on $\R^3$, where $0$ stands for the mean vector, and $I$, for the $3\times 3$ identity matrix  for the variance. In the following "density of  measures"  shall be understood relative to the underlying Lebesgue measure.
\begin{thm}\label{invsol} Let us assume that the  law of the initial velocity $Z_0$ and that of the initial location $X_0$ of  \eqref{eq-vel},\eqref{eq-space}  are independent. Let $Z_0$ have  MVN $(0, I)$ distribution.  Assume that the distribution $\eta_0$ of the  initial location $X_0$ has density $h(x)$, $x\in$ $\mathbb{R}^3$.  Then 
 the joint distribution  $\mu(dx, dz)$ of $\{X_s, Z_s\}_{s\in \mathbb{R^+}}$  has for all $t \ge 0$ density $\rho_t(x,y)$$:=h_t(x)g(y)$, where  $ g(y)$ is the density of the normal distribution MVN $(0, I)$, while $h_t(x)$ is the density of $X_t$.
\end{thm}

\begin{Remark} In particular   the  marginal velocity  $Z_t$ at time $t$  is distributed according to the MVN $(0, I)$ distribution for all $t \ge 0$, and is independent of the location $X_t$ for all $t \ge 0$. 
\end{Remark}

\begin{proof} Our method of proof relies on guessing the solution and proving that it is indeed the solution. We take 
\begin{align} \label{ind}
\mu_t(dz, dx) = \mu_t(dz\,|\,x) \eta_t(x).
\end{align}
with 
\begin{align} \mu_t(dz\,|\,x):=g(z) dz \quad\forall t\geq 0\end {align} and 
\begin{align}  \label{Gauss} \eta_t(dx):= h_t(x) dx \quad \forall t\geq 0,\end {align} where 
$h_t(x)$ is a probability  density function on $\mathbb{R}^3$ which is ascertained below. We will then  prove that  $\mu_t(dz, dx)$ is the distribution of a process $\{X_s, Z_s\}_{s\in \mathbb{R^+}}$ which solves   \eqref{eq-vel},\eqref{eq-space} with 
$Z_s$ having the  same distribution as $Z_0$ and solving  \eqref{eq-vel}. It would then follow that  $\int_0^t Z_s ds$ is  a Gaussian random variable  for all $t\geq 0$, and $X_t$  has therefore a density function denoted by  $h_t(x)$.

\par 
Consider $\phi_t(\l ): = \mathbb{E} \left[e^{i(\l , Z_t)}\right]$ for any $\l \in \R^3$, and $t \ge 0$.  It is enough to  prove that $\phi_t(\l ) = \phi_0(\l )$ for all $\l$. Using the It\^o formula and taking expectation, one obtains 
\begin{align}
\phi_t(\l ) &= \phi_0(\l ) + \int_0^t \int_{\mathbb{D}\times \mathbb{D} \times \Xi} \left\{e^{i(\l , z_s + \a (z_s, v_s, \xi))} - e^{i(\l , z_s)}\right\}\notag\\
&\;\;\;\;\;\;\;\;\;\;\;\;\s (|z_s - v_s|)\beta(|x-y|)\mu_s(dz, dx) \mu_s(dv, dy)Q(d\xi)ds.
\label{eqn4.1}
\end{align}
This is an equation that is satisfied by the characteristic function of $Z_t$ where $Z$ is a solution of  \eqref{eq-vel}. If $\mu_t(dz, dx)$ is as specified above, then we can write 
\begin{align}
\phi_t(\l) &= \phi_0(\l ) + \int_0^t  \int_{\mathbb{R}^6\times \mathbb{R}^6 \times \Xi} \left\{e^{i(\l , z_s + \a (z_s, v_s, \xi))} - e^{i(\l , z_s)}\right\}
\notag\\
& \;\;\s (|z_s - v_s|)\beta(|x - y|) g(z) h_s(x) dz dx g(v) h_s(y) dv dyQ(d\xi)ds
\label{eqn4.2}
\end{align}
which we write as $ \phi_0(\l) + I.$
Let us  write $I$ as 
\begin{align*}
I &=  \int_0^t  \int_{\mathbb{R}^6 \times \Xi} \phi(s, x, y, \xi) h_s(x) h_s(y) Q(d\xi)dxdyds
\end{align*}
where \[
\phi(s, x, y, \xi) := \int_{\R^6}\left\{e^{i(\l , z_s + \a (z_s, v_s, \xi))} - e^{i(\l , z_s)}\right\}\s (|z_s - v_s|)\beta(|x - y|) g(z) dz g(v) dv.
\] 
 Then $\phi(s, x, y, \xi)$ is
\begin{align*}
& = \int_{\mathbb{R}^3\times \mathbb{R}^3} (e^{i(\l , z_s + \a (z_s, v_s, \xi))}- e^{i(\l , z_s)})\s (|z_s - v_s|)\beta(|x-y|)\mu_s(dz\,|\,x) \mu_s(dv\,|\,y)\\
& = \frac{\beta(|x-y|)}{(2\pi)^{3/2}}\int_{\mathbb{R}^3\times \mathbb{R}^3} (e^{i(\l , z^*_s)}- e^{i(\l , z_s)})\s (|z^*_s - v^*_s|) \exp\{-1/2(|z_s|^2 + |v_s|^2)\}dz dv\\
&= \frac{\beta(|x-y|)}{(2\pi)^{3/2}}\int_{\mathbb{R}^3\times \mathbb{R}^3} (e^{i(\l , z_s)}-e^{i(\l , z^*_s)})\s (|z_s - v_s|) \exp\{-1/2(|z^*_s|^2 + |v^*_s|^2)\}dz dv
\end{align*}
by Proposition \ref{PropAppTanaka}; continuing, 
\begin{align}
& = \frac{\beta(|x-y|)}{(2\pi)^{3/2}}\int_{\mathbb{R}^3\times \mathbb{R}^3} (e^{i(\l , z_s)}-e^{i(\l , z^*_s)})\s (|z_s - v_s|) \exp\{-1/2(|z_s|^2 + |v_s|^2)\}dz dv\notag\\
&=-\int_{\mathbb{R}^3\times \mathbb{R}^3} (e^{i(\l , z_s + \a (z_s, v_s, \xi))}- e^{i(\l , z_s)})\s (|z_s - v_s|)\beta(|x-y|)g(z)dz g(v)dv 
\label{eqn4.3}
\end{align}
by using conservation of energy in \eqref{eqn1.3}, and with 
 $z^*$ and $v^*$ denoting  the post-collision velocities corresponding to the pre-collision velocities $z$ and $v$. It follows that  $I=-I$ and hence  $I=0$ so  that $\phi_t(\l ) = \phi_0(\l )$ for all $\l$  
at all times $t > 0$. 
\end{proof}
\par

\medskip
\noindent
{\bf Acknowledgements}: We are very grateful to Professors Anna de Masi, Alessandro Pellegrinotti, Errico Presutti, and Mario Pulvirenti for illuminating discussions and references during the conference "Interacting particle systems in thermodynamic models", 26-30 January 2015, at the Gran Sasso Science Institute (GSSI)  at L'Aquila sponsored by GSSI and the German Science Foundation DFG. We thank Martin Friesen for useful discussions in  the revised version. Last but not least we thank  the anonymous referee for pointing an error in the original manuscript, which led to a substantial improvement of this article. The support of the Hausdorff Center of Mathematics, University of Bonn, the Mathematics Department of Louisiana State University  and the Stochastic Group of the Bergische Universit\"at of  Wuppertal is also gratefully acknowledged.

\end{document}